\documentclass[11pt,reqno]{amsart}
\usepackage{amsmath,amsfonts,amstext,amssymb,amsbsy,amsopn,amsthm,eucal}
\usepackage{txfonts}
\usepackage{dsfont}
\usepackage{graphicx}   
\usepackage{hyperref}
\usepackage{color}
\usepackage{verbatim}

\newcommand{\Rm}{\text{Rm}}
\newcommand{\Ric}{\text{Ric}}

\newcommand{\Vol}{\text{Vol}}

\newcommand{\dN}{\mathds{N}}

\newcommand{\dQ}{\mathds{Q}}
\newcommand{\dR}{\mathds{R}}

\newcommand{\dZ}{\mathds{Z}}

\newcommand{\Hess}{\text{Hess}}

\newcommand{\R}{\text{R}}

\newcommand{\cA}{\mathcal{A}}
\newcommand{\cB}{\mathcal{B}}

\newtheorem{theorem}{Theorem}[section]

\newtheorem{proposition}[theorem]{Proposition}
\newtheorem{lemma}[theorem]{Lemma}
\newtheorem{corollary}[theorem]{Corollary}
\newtheorem{question}{Question}[section]
\theoremstyle{definition}
\newtheorem{definition}[theorem]{Definition}
\theoremstyle{remark}
\newtheorem{remark}{Remark}[section]
\theoremstyle{remark}
\newtheorem{example}{Example}[section]
\theoremstyle{remark}

\theoremstyle{remark}

\begin{document}

\title{Six Dimensional Counterexample to the Milnor Conjecture}

\author{Elia Bru\`e, Aaron Naber and Daniele Semola}

\date{\today}

\begin{abstract}
We extend the previous work of \cite{BrueNaberSemolaMilnor1} by building a smooth complete manifold $(M^6,g,p)$ with $\Ric\geq 0$ and whose fundamental group $\pi_1(M^6)=\dQ/\dZ$ is infinitely generated.  The example is built with a variety of interesting geometric properties.  To begin the universal cover $\tilde M^6$ is diffeomorphic to $S^3\times \dR^3$, which turns out to be rather subtle as this diffeomorphism is increasingly twisting at infinity.  The curvature of $M^6$ is uniformly bounded, and in fact decaying polynomially.  The example is {\it locally} noncollapsed, in that $\Vol(B_1(x))>v>0$ for all $x\in M$.  Finally, the space is built so that it is {\it almost } globally noncollapsed.  Precisely, for every $\eta>0$ there exists radii $r_j\to \infty$ such that $\Vol(B_{r_j}(p))\geq r_j^{6-\eta}$. 

The broad outline for the construction of the example will closely follow the scheme introduced in \cite{BrueNaberSemolaMilnor1}.  The six-dimensional case requires a couple of new points, in particular the corresponding Ricci curvature control on the equivariant mapping class group is harder and cannot be done in the same manner.
\end{abstract}

\maketitle

\section{Introduction}

The main result of this paper is the existence of a six-dimensional smooth complete Riemannian manifold with nonnegative Ricci curvature and infinitely generated fundamental group:

\begin{theorem}\label{t:milnor6d}
	There exists a smooth complete Riemannian manifold $(M^6,g,p)$ with $\Ric\geq 0$, $\pi_1(M)=\dQ/\dZ$ and such that the universal cover $\tilde M^6$ is diffeomorphic to $S^3\times \dR^3$. 
	Further, for each $\eta>0$ the example may be constructed so that it satisfies
\begin{enumerate}
	\item $|\Rm|(x)\leq \frac{C(\eta)}{d(p,x)^{2-\eta}}$ for every $x\in M$,
	\item There exists $r_j\to \infty$ such that $\Vol(B_{r_j}(p)) \geq r_j^{6-\eta}$ , 
	\item There exists $r_j\to \infty$ such that $\Vol(B_{r_j}(p)) \leq r_j^{3+\eta}$ , 
	\item $\Vol(B_1(x))>v>0$ for all $x\in M$ .
\end{enumerate}
\end{theorem}
\begin{remark}
	We point out that an argument analogous to the one used in this paper shows that the examples constructed in \cite{BrueNaberSemolaMilnor1} can be taken to have universal cover diffeomorphic to $S^3\times\dR^4$. 
\end{remark}
\begin{remark}
	Note that once an example with $\pi_1(M)=\dQ/\dZ$ is constructed we can automatically construct an example with $\pi_1(N)=\Gamma\leq \dQ/\dZ$ for any subgroup by looking at $N\equiv \tilde M/\Gamma$ .
\end{remark}

This provides a counterexample to a conjecture of Milnor \cite{Milnor} in a dimension lower dimension than our previous paper \cite[Theorem 1.1]{BrueNaberSemolaMilnor1}, where we built a seven-dimensional manifold with nonnegative Ricci curvature and infinitely generated fundamental group. We address the reader to the introduction of \cite{BrueNaberSemolaMilnor1} and to the survey papers \cite{ShenSormani} and \cite{Pansurvey} for a detailed bibliography and a discussion on the previous positive results about the Milnor conjecture and the known properties of fundamental groups of manifolds with $\Ric\ge 0$.
\\

The statement of Theorem \ref{t:milnor6d}.2 above should be compared with a result of B.-Y. Wu (see \cite{Wu}) saying that for $\alpha\leq \alpha(n)$ if $(M^n,g)$ has $\Ric\ge 0$ and the limit
\begin{equation}
\lim_{r\to \infty}\frac{\mathrm{vol}(B_r(p))}{r^{n-\alpha}}
\end{equation} 
exists and is strictly positive, then $\pi_1(M)$ is finitely generated. The effect of Theorem \ref{t:milnor6d}.2 is to show that the limit in the assumptions of \cite[Theorem 1.2]{Wu} cannot be replaced by a limsup.\\

The broad outline for the construction of the example in Theorem \ref{t:milnor6d} will closely follow the scheme introduced in \cite{BrueNaberSemolaMilnor1}.  The six dimensional case requires a couple new points, and in particular the corresponding Ricci curvature control on equivariant mapping class group is much harder.  We discuss this below.\\

As in \cite{BrueNaberSemolaMilnor1}, we will work at the level of the universal covering space $\tilde M$. 
A key step of the construction of the $7$-dimensional examples was the existence for each $k\in\dZ$ of a smooth family of Riemannian metrics $(S^3 \times S^3 , g_t)$ having some key properties.  The family began with the standard metric $g_0 = g_{S^3_1 \times S^3_1}$ and all had positive Ricci curvature.  Additionally, each metric $g_t$ is invariant under the left $(1,k)$ Hopf action:
\begin{align}
	\theta\cdot_{(1,k)}(s_1,s_2) 
	=
	(e^{i\theta}\cdot s_1,e^{ik\theta}\cdot s_2)\, ,
	\quad
	\theta \in S^1\, ,  \,  s_1,s_2 \in S^3 \, ,
\end{align}
and is such that $g_1 = \phi^* g_0$ for some diffeomorphism $\phi : S^3 \times S^3 \to S^3 \times S^3$ conjugating the $(1,k)$ to the $(1,0)$ action:
\begin{equation}
	\phi(\theta \cdot_{(1,k)}(s_1,s_2)) = \theta \cdot_{(1,0)} \phi(s_1,s_2) \, , \quad \text{for every $s_1,s_2\in S^3$}\, .
\end{equation}
See \cite[Section 6]{BrueNaberSemolaMilnor1} for more details. This family was used as a family of cross-sections in some annular regions on $\tilde M$, hence the resulting dimension was seven.\\

In order to build a six-dimensional example, we will replace the family of six-dimensional cross-sections discussed above with a five-dimensional family with analogous properties. To this aim, we define the left $(1,k)$ action on $S^3 \times S^2$ as
\begin{equation}
	\theta\cdot _{(1,k)}(s_1,s_2):= (e^{i\theta} \cdot s_1, e^{ik\theta}\cdot  s_2)\, , \quad \theta \in S^1 \, ,
\end{equation}
where $e^{i\theta }\cdot s_1$ indicates the left Hopf rotation in $S^3$ and $e^{ik\theta} \cdot s_2 := (e^{ik\theta}z,t)$, where we identify $s_2 = (z,t)\in S^2 \subset \mathbb{C} \times \dR$.

The key new contribution of the present paper is Theorem \ref{t:equivariant_mapping_class_S3xS2} below.  We show that when $k$ is even\footnote{The restriction to even $k$'s is topological, and not related to the existence of the family of metrics with positive Ricci curvature. Indeed, for $k$ odd, the $(1,k)$ and $(1,0)$ actions are not conjugated, see Lemma \ref{lemma:diffeostruct}.} there exists a smooth family of positively Ricci curved Riemannian metrics $(S^3\times S^2, g_t)$, $t\in [0,1]$, invariant with respect to the $(1,k)$-action, and such that $g_0 = g_{S^3\times S^2}$ and $g_1 = \phi^* g_0$, where $\phi : S^3 \times S^2 \to S^3 \times S^2$ is a diffeomorphism satisfying
\begin{equation}
	\phi(\theta\cdot_{(1,k)}(s_1,s_2)) = \theta \cdot _{(1,0)}\phi(s_1,s_2)\, , \quad \theta \in S^1\, ,\quad (s_1,s_2)\in S^3\times S^2\, .
\end{equation}

The construction of this family introduces some new challenges with respect to the analogous construction in \cite{BrueNaberSemolaMilnor1}, in particular some steps of the construction from \cite{BrueNaberSemolaMilnor1} necessarily fail.  In spirit this is because the $(1,k)$ action on $S^3\times S^2$ is more wild and less homogeneous in nature.  A little more precisely, and relying a little unfairly on the readers knowledge of \cite{BrueNaberSemolaMilnor1} with the understanding that this will be explained better later, when viewing $(S^3\times S^2,g_t)\stackrel{S^1}{\longrightarrow}(N,h_t)$ as a circle bundle over some underlying space, we cannot hope to equip this bundle with a family of Yang-Mills connections.  This consequently adds some rather dramatic Ricci curvature terms, and controlling them is quite delicate and requires something with a new flavor to it.  See Section \ref{subsec:outtwisting} for details.\\

We conclude the introduction with a list of open questions related to the Milnor conjecture that seem worthwhile for future investigation, without the aim of being complete.\\

In view of the existing positive results in dimension less or equal than $3$ (see \cite{CohnVossen} and \cite{Liu}) and of the counterexamples constructed in the present paper, the validity of the Milnor conjecture remains an open question in dimensions $4$ or $5$.

\begin{question}
Let $(M^n,g)$ be a complete with $\Ric\geq 0$ and $n=4$ or $n=5$. Is $\pi_1(M)$ finitely generated?
\end{question}

We believe that the construction of a counterexample to the Milnor conjecture in dimension $4$ or $5$, if it exists, would most likely require the development of a different strategy with respect to the one employed in the present paper and in our previous \cite{BrueNaberSemolaMilnor1}. \\

Our constructions are necessarily nonnegative Ricci in nature, however note by Theorem \ref{t:milnor6d}.1 we have that the current example has polynomially decaying Ricci curvature.  This leads to the question:

\begin{question}
	Let $(M^n,g)$ be complete with $\Ric\equiv 0$. Is $\pi_1(M)$ finitely generated?
\end{question}

To the best of our knowledge, the validity of the Milnor conjecture is open also in the K\"ahler case:

\begin{question}
Let $(M^{2n},g,J)$ be a complete K\"ahler manifold with $\Ric\ge 0$. Is $\pi_1(M)$ finitely generated?
\end{question}

As we already remarked in \cite{BrueNaberSemolaMilnor1}, the universal covers of the counterexamples to the Milnor conjecture that we can construct have less than Euclidean volume growth.  The context where $\tilde M$ is noncollapsed is still open:

\begin{question}\label{q:AVRMilnor}
Let $(M,g)$ be a complete manifold with $\Ric\ge 0$ such that the universal cover $(\tilde{M},\tilde{g})$ satisfies the noncollapsing condition $\Vol(B_r(\tilde p))>v r^n$ for all $r>0$. Is $\pi_1(M)$ finitely generated?
\end{question}

We address the reader to \cite{Pan}, \cite{PanRong}, and \cite{Panb} for some recent interesting partial positive results about Question \ref{q:AVRMilnor}.\\

A theorem of Wei \cite{Wei} says that any finitely generated torsion-free nilpotent group is the fundamental group of some complete $(M^n,g)$ with $\Ric\ge 0$. 
It is an open question whether infinite generation of fundamental groups for manifolds with $\Ric\ge 0$ is a purely abelian phenomenon. In particular, we have the following:

\begin{question}
Let $(N,\cdot)$ be a simply connected nilpotent Lie group, which we view as a matrix group $N<\mathrm{GL}(m,\dR)$ for some $m\in\dN$. Does there exist a complete Riemannian manifold $(M^n,g)$ with $\Ric\ge 0$ and such that $\pi_1(M)=N\cap \mathrm{GL}(m,\dQ)$?
\end{question}

The remainder of this paper is organized as follows: \\

In Section \ref{s:outline_milnor:inductive} we describe the inductive construction for the counterexamples. The strategy is completely analogous to the one introduced in our previous paper \cite{BrueNaberSemolaMilnor1}, and it is based on three main steps. Given the existence of the equivariant family of positively Ricci curved metrics with the properties discussed above, in particular see Theorem \ref{t:equivariant_mapping_class_S3xS2}, the three main inductive propositions can be proved arguing as for the corresponding statements in \cite{BrueNaberSemolaMilnor1}.  Hence their proofs will be omitted.\\

In Section \ref{sec:univdiffeo} we prove that the counterexamples can be constructed so that their universal covers are diffeomorphic to $S^3\times\dR^3$. The proof will be based on the construction of a Morse-Bott function and it requires a careful analysis of the gluing diffeomorphisms in the inductive construction.\\

In Section \ref{sec:furtherproperties} we discuss the curvature decay and the volume growth of the examples.  Section \ref{s:preliminaries} records some useful preliminary material.
Section \ref{sec:equivtwist} is dedicated to the proof of Theorem \ref{t:equivariant_mapping_class_S3xS2}.

\section*{Acknowledgments}
The first author would like to express gratitude for the financial support received from Bocconi University.\\
The third author is supported by FIM-ETH with a Hermann-Weyl Instructorship. He is grateful to Vitali Kapovitch and to the members of the Geometry group of the University of M\"unster for some helpful conversations about the subject of the paper.
\\
The authors are grateful to the anonymous referees for several useful suggestions.

\vspace{.5cm}

\section{Inductive Construction for Theorem \ref{t:milnor6d}}\label{s:outline_milnor:inductive}

The inductive construction in the proof of Theorem \ref{t:milnor6d} is essentially borrowed from our previous paper \cite{BrueNaberSemolaMilnor1}.  We will stress those points where the construction is either different or more refined, as happens at several points in order to address the geometric and topological properties of our example.  We outline the construction in this Section for ease of readability and in order to record some notation that will be helpful in the future sections.

\subsubsection{\bf Decomposing the group $\Gamma\leq \dQ/\dZ$}\label{ss:outline_milnor_construction:decomposition} Let us begin by choosing in $\Gamma\leq \dQ/\dZ\subseteq S^1$ a nested sequence of
finitely generated subgroups $\{e\}=\Gamma_{-1}\leq \Gamma_0\leq \Gamma_1\leq \cdots$ which generate $\Gamma$ in the sense that for every $\gamma\in \Gamma$ we have that $\gamma\in \Gamma_j$ for some $j$ sufficiently large.  For instance such a sequence of subgroups may be built using that $\Gamma$ is countable and choosing an enumeration.  A finitely generated subgroup $\Gamma_j\leq \dQ/\dZ$ is necessarily finite and generated by a single element $\gamma_j\in \Gamma_j$.  In this way we can write 
\begin{align}\label{e:outline_milnor:Gamma}
	\Gamma_j = \big\langle \gamma_j, \Gamma_{j-1}\big\rangle \text{ and $\exists!$ minimal $k_j\in \dN$ such that } \gamma_j^{k_j}=\gamma_{j-1}\, .
\end{align}
It will be convenient to adopt the notation $k_{\le j}\equiv k_0\cdot k_1\cdots\cdot k_j$, for $j\in \mathbb{N}$ and we shall denote by $|\gamma|$ the order of any $\gamma\in\Gamma$.  Notice that, with this notation, $|\gamma_j|=k_{\le j}$.  There is no harm in assuming that $k_j>1$ for each $j$, as otherwise $\Gamma_{j}=\Gamma_{j-1}$. \\

The only additional requirement that we make, with respect to the analogous construction in \cite{BrueNaberSemolaMilnor1}, is that we will assume throughout that $k_0$ is even. It is clear that one can find a decomposition of $\Gamma=\dQ/\dZ$ as above compatible with this restriction. Moreover, it is also obvious that if there exists a complete $(M^6,g)$ with $\Ric\ge 0$ and $\pi_1(M)=\dQ/\dZ$, then there exists also a complete $(N^6,g)$ with $\Ric\ge 0$ and $\pi_1(N)=\Gamma$, for any $\Gamma<\dQ/\dZ$.\\

For each $\gamma\in \Gamma$ we can then uniquely write it as

\begin{align}
	\gamma = \prod_{j} \gamma_j^{a_j}\, , \text{ such that }a_j<k_j\, ,
\end{align}
where at most a finite number of $a_j$ are nonvanishing.  Note that there is the short exact sequence $0\to \Gamma_j\to \Gamma\to \Gamma/\Gamma_j\to 0$.  This does not split as a group splitting of course, however the choice of basis builds for us a splitting of sets
\begin{align}\label{e:Gamma_splitting}
	&\Gamma = \Gamma_j\oplus \Gamma/\Gamma_j\, ,\text{ given by }\notag\\
	&\gamma = \gamma_{\leq j}\cdot\gamma_{>j} = \prod_{i\leq j} \gamma_i^{a_i}\cdot \prod_{i>j} \gamma_i^{a_i}\, .
\end{align}

\begin{remark}
It is possible, and helpful, to include into the discussion the case where $\Gamma$ is finitely generated, or equivalently $\Gamma=\Gamma_j$ for some $j$.  This is more in line with how our inductive construction in Section \ref{s:outline_milnor:inductive} will proceed.  
\end{remark}

Before moving on with the construction, let us introduce some helpful notation. Given $(a,b)\in \dZ\times \dZ$ we endow $S^3 \times S^2$ with the (left) $(a,b)$-action
\begin{equation}
	\theta\cdot _{(a,b)}(g,s):= (e^{ia\theta} \cdot g, e^{ib\theta}\cdot  s)\, , \quad \theta \in S^1 \, ,
\end{equation}
where $e^{i\theta }\cdot g$ indicates the left Hopf rotation in $S^3$, and $e^{ib\theta} \cdot s := (e^{ib\theta}z,t)$, where we identify $s = (z,t)\in S^2 \subset \mathbb{C} \times \dR$.\\

\subsection{The inductive construction}\label{ss:inductiveconstruction}

The proof of Theorem \ref{t:milnor6d} will be set up in an inductive fashion, where we will build a sequence of pointed manifolds $(M_j,p_j,\Gamma_j)$ with $\Ric_j\geq 0$ together with free uniformly discrete isometric actions by $\Gamma_j$.  This Section will begin with a description of the main properties of our inductive sequence $M_j$, together with how one proves Theorem \ref{t:milnor6d} once this sequence has been constructed.  The induction criteria will be such that building $(\tilde M,\Gamma)$ from the inductive sequence will be relatively straightforward.\\  

The remainder of this Section will then focus on proving the induction, namely on how to construct $M_{j+1}$ from $M_j$ in order to complete the induction proof.  The construction will boil down to three major steps, and in each step we will state one of our three main inductive Propositions. The proof of Theorem \ref{t:milnor6d} will be complete in this Section modulo these main Propositions.\\

The main inductive Propositions will correspond to the analogous statements in our previous paper, see \cite[Section 3]{BrueNaberSemolaMilnor1}. Two of them, namely the construction of the model space in Proposition \ref{p:step1}, and the gluing construction with action extension in Proposition \ref{p:step3}, can be proved with very minor changes with respect to the proofs of the corresponding statements from \cite{BrueNaberSemolaMilnor1}. Hence their proofs will be omitted.\\ 
Also the remaining inductive Proposition \ref{p:step2} can be proved with very minor changes with respect to the corresponding statement from \cite{BrueNaberSemolaMilnor1}, once the twisting Theorem \ref{t:equivariant_mapping_class_S3xS2} has been proved. The proof of Proposition \ref{p:step2} will be omitted accordingly. On the other hand, the proof of Theorem \ref{t:equivariant_mapping_class_S3xS2}, which constitutes the main novelty of the present paper, is deferred to Section \ref{sec:equivtwist} below.

The proof of the induction given the main Propositions is completely analogous to the one in \cite{BrueNaberSemolaMilnor1}. However, we report it in this Section for the ease of readability.\\

Recall that we have chosen as in \eqref{e:outline_milnor:Gamma} a sequence of finitely generated subgroups $\Gamma_j\leq \Gamma$ with $\Gamma_j = \langle \gamma_j,\Gamma_{j-1}\rangle$ which are all generated by a single element $\gamma_j$ such that $\gamma_j^{k_j}=\gamma_{j-1}$.

Our geometric construction will be based on a sequence of parameters $\epsilon_j\to 0$ and $\delta_j\to 0$.  We may begin by choosing any sequence $\epsilon_j\to 0$.  Indeed any sequence of constants $\epsilon_j<1$ will do, but in our description of the tangent cones at infinity of $\tilde M$ later in Section \ref{ss:tangentcones} we will use that these constants tend to zero, which gives a slightly cleaner picture. Let $\delta_1\ll1$ also be any small constant, the remaining $\delta_j$ will be chosen based on applications of our Inductive Propositions. We shall adopt the notation $A_{s_1,s_2}(p)$ to denote the annulus $B_{s_2}(p)\setminus B_{s_1}(p)$ for any $0\le s_1<s_2\le \infty$.

Our sequence $(M_j,p_j,\Gamma_j)$ will inductively be assumed to satisfy:

\begin{enumerate}
	\item[(I1)] There exists a free isometric action by $\Gamma_j$ on $M_j$ with $r_j\equiv d(p_j,\gamma_j\cdot p_j)$ and $\frac{r_j}{k_{j-1} r_{j-1}} >> 1 $.
	\item[(I2)] There exists an isometry $\Phi_j:U_j\subseteq M_j\to M_{j+1}$ with $B_{10 k_j r_j}(p_j)\subseteq U_j\subseteq B_{10^3 k_j r_j}(p_j)$ with $\Phi_j(p_j)=p_{j+1}$, where $U_j$ is $\Gamma_j$ invariant with $\Phi_j(\gamma\cdot x)=\gamma\cdot \Phi_j(x)$ for all $\gamma\in \Gamma_j\leq \Gamma_{j+1}$ .
	\item[(I3)] $M_j\setminus U_j$ is isometric to $S^3_{\delta_j r_j}\times A_{10^{2}k_jr_j,\infty}(0)\subseteq S^3_{\delta_j r_j}\times C(S^2_{1-\epsilon_j})$ \footnote{Observe that this is isometrically very close to $S^3\times \dR^3$.  Indeed, in our setup $U_j$ itself is very Gromov-Hausdorff close to $S^3\times \dR^3$.}.  The action of $\gamma_j$ in this domain rotates the cross section $S^2_{1-\epsilon_j}$ of the cone factor by $2\pi/k_j$ and Hopf rotates the $S^3_{\delta_j r_j}$ factor by $2\pi/|\gamma_j|=2\pi/(k_0k_1\cdots k_j)$.
\end{enumerate} 
\vspace{.2cm}
\begin{remark}
	It follows from (I3) that the orbit of the base-point with respect to the action of $\Gamma_j$ has diameter roughly $k_j r_j$.
\end{remark}
\begin{remark}
	It will be clear from the construction that $\frac{r_{j+1}}{k_{j} r_{j}} \to \infty$.  That is, the scale of the action of the next generator $\gamma_{j+1}$ relative to the orbit of the previous generator $\gamma_j$ is tending to infinity.
\end{remark}

Before discussing more carefully the structure of the spaces $M_j$ above, let us quickly see that if such an inductive sequence as above can be built, then we are done.  Indeed, consider first the $\Gamma_i$-equivariant isometries $\Phi_{ji}=\Phi_j\circ\cdots\circ\Phi_i:U_i\to U_{ji}\equiv \Phi_{ji}(U_i)\subseteq M_j$.  We can take an abstract equivariant pointed Gromov-Hausdorff limit of the sequence $(M_j,p_j,\Gamma_j)$. However the setup is such that we can also simply define the direct limit
\begin{align}
	\tilde M\equiv\big\{ (x_j,x_{j+1},\ldots): x_{k+1}=\Phi_k(x_k) \text{ for all }k\geq j\big\}/\sim\, ,
\end{align}
where there is an equivalence relation $(x_{j},x_{j+1},\ldots)\sim (y_{j'},y_{j'+1},\ldots)$ if there exists $k\geq \max\{j,j'\}$ such that $x_k=y_k$.   By the equivariance of the isometries $\Phi_{i}$ we have that $\Gamma_j$ naturally acts on all sequences $(x_k,x_{k+1},\ldots)$ with $k\geq j$.  In particular there is an induced action of $\Gamma$ on $\tilde M$.  Note that $U_j\subseteq M_j$ all embed isometrically into $\tilde M$ and exhaust $\tilde M$, and the restriction of the $\Gamma_j$ action to $U_j\subseteq \tilde M$ is the expected action.   Thus $\tilde M$ is a smooth, simply connected Riemannian manifold with $\Ric\geq 0$ and a free discrete isometric action by $\Gamma$, as claimed.

\vspace{.2cm}
\subsection{The Steps of the Inductive Construction}

We will break down this inductive construction into three steps.  Each will involve a Proposition which will form the main constructive ingredient in the step.  Our goal in this subsection is then to discuss these steps and state the Propositions.  We will then see how to finish the induction given these results.\\  

The first step will build our background model space $\cB(\epsilon,\delta)\approx S^3\times \dR^3$.  It will form the basis of both our base step of the induction, and also the underlying space for which previous induction manifolds $M_j$ will be glued into in order to form $M_{j+1}$.\\  

Each $M_j$ looks like $S^3\times \dR^3$ at infinity with the action $\Gamma_j$ induced by the $(1,k_{\le j-1})$ action.  The first step in building $M_{j+1}$ is to equivariantly twist $M_j$ to a new manifold $\hat M_j$, so that after our twisting $\hat M_j$ again looks like $S^3\times \dR^3$ at infinity but this time the $\Gamma_j$ action is induced by the $(1,0)$-action. This will be the goal of Step 2.\\

The third step of the inductive construction is to take our twisted $\hat M_j$ and glue in $k_{j+1}$ copies into a new base manifold $\cB_{j+1}$.  The gluing is such that we have now extended the $\Gamma_j$ action on $\hat M_j$ to a $\Gamma_{j+1}=\langle \gamma_{j+1},\Gamma_j\rangle$ action on $M_{j+1}$ in the appropriate fashion.

\vspace{.25cm}

\subsubsection{\bf Step 1: The Background Model Space $\cB(\epsilon,\delta)$}

Our construction will begin by building a background manifold $\cB(\epsilon,\delta)$.  The space will both play the role of base step in the inductive construction and additionally when we move from $M_j$ to $M_{j+1}$ the basis for our construction will be to glue in $k_{j+1}$ copies of $M_j$ into the background space $\cB_{j+1}$.  \\

The construction of $\cB(\epsilon,\delta)$ is relatively straightforward, we will simply take $S^3\times \dR^3$ and slightly curve the $\dR^3$ factor in order to give it a slight cone angle.  The precise setup is the following:

\begin{proposition}[Step 1: The Model Space]\label{p:step1}
	For each $\delta>0$ and $1>\epsilon>0$ , there exists a smooth manifold $\cB^6=\cB(\epsilon,\delta)$ such that the following hold:
	\begin{enumerate}
		\item $(\cB^6,g_B,p)$ is a complete Riemannian manifold with $\Ric\geq 0$, and it is diffeomorphic to $S^3\times \dR^3$.
		\item There exists $B_{10^{-3}}(p)\subseteq U\subseteq B_{10^{-1}}(p)$ such that $\cB\setminus U$ is isometric to $S^3_{\delta}\times A_{10^{-2},\infty}(0)\subseteq S^3_{\delta}\times C(S^2_{1-\epsilon})$ 
		\item There is an isometric $T^2=S^1\times S^1$ action on $\cB$ for which on $\cB\setminus U\approx S^3_{\delta}\times C(S^2_{1-\epsilon})$ the first $S^1$ acts on the $S^3_{\delta}$ factor by a globally free (left) Hopf rotation and the second $S^1$ acts on the cross sections $S^2_{1-\epsilon}$ of the cone factor by rotation. 
		\item The $S^1$-action induced by the homomorphic embedding $S^1\ni \theta\mapsto (a\theta,b\theta)\in T^2$ is free whenever $(a,b)\in \dZ\times \dZ$ are coprime and $a\neq 0$.
	\end{enumerate}
\end{proposition}
\begin{remark}
	Thus for each $(a,b)\in \dZ\times \dZ$ we have the induced $(a,b)$-$S^1$ action given by the homomorphic embedding $S^1\ni \theta\mapsto (a\theta,b\theta)\in T^2$.
\end{remark}

The proof of Proposition \ref{p:step1} is completely analogous to the proof of the corresponding Proposition in our previous paper, see \cite[Section 5]{BrueNaberSemolaMilnor1}. Hence it will be omitted.\\

{\bf Base Step: } Let us then define the base step of our induction as $M_1=\cB(\epsilon_1,\delta_1)$ as above.  We will equip $M_1$ with the isometric group action of $\Gamma_1\leq S^1$, which is induced by the $(1,k_{0})$-action as above.  In particular, on $S^3_{\delta_1}\times S^2_{1-\epsilon_1}$ we have that the generator $\gamma_1$ will act by Hopf rotating $S^3_{\delta_1}$ by $2\pi/|\gamma_1| = 2\pi/(k_1k_0)$ and by rotating the cross section of $C(S^2_{1-\epsilon_1})$ by $2\pi/k_1=2\pi k_0/|\gamma_1|$.\\

\vspace{.3cm}

\subsubsection{\bf Step 2:  Twisting the Geometry of $M_j$ at Infinity}\label{ss_Step2description}

By condition (I3) of the induction we know that outside some compact set $U_j$ our space $M_j\setminus U_j$ is isometric to $S^3_{\delta_j r_j}\times A_{10^{2}k_jr_j,\infty}(0)\subseteq S^3_{\delta_j r_j}\times C(S^2_{1-\epsilon_j})\approx S^3\times \dR^3$.  Further, we understand that in this region the action of the generator $\gamma_j\in \Gamma_j$ looks primarily like a rotation of the $\dR^3$ factor.  More precisely, it rotates the $\dR^3$ factor by $2\pi/k_j$ and it Hopf rotates the $S^3_{\delta_j r_j}$ factor by the much smaller $2\pi/|\gamma_j| = 2\pi/(k_0\cdots k_j)$. \\

In Step 3 we will be gluing $k_{j+1}$ copies of $M_j$ into a model space $\cB_{j+1}$, and in the gluing region we will again have that $\cB_{j+1}\approx S^3\times \dR^3$.  However, the action of $\Gamma_j$ on $\cB_{j+1}$ will look like a rotation of just the $S^3$ factor without any rotational bit on the $\dR^3$ factor.  Thus to accomplish the gluing we will need to modify $M_j$ at infinity into a new space $\hat M_j$, which will again look close to $S^3\times \dR^3$ but for which the action of $\gamma_j$ is now purely a rotation of the $S^3$ factor.\\

In order to address this problem let us first consider $S^3\times S^2$ with the standard product metric $g_{S^3\times S^2}$, and let us recall that if $(a,b)\in \dZ\times \dZ$ then we have the $S^1$-isometric action $\cdot_{(a,b)}: S^1\times S^3\times S^2\to S^3\times S^2$ which acts by $a$ times the (left) Hopf rotation on the $S^3$ factor and $b$ times the rotation with respect to a fixed axis on the $S^2$ factor.  The following will describe how the cross sections of our new space $\hat M_j$ will be twisted.  It will be proved in Section \ref{sec:equivtwist}:\\

\begin{theorem}\label{t:equivariant_mapping_class_S3xS2}
	Let $g_0=g_{S^3\times S^2}$ be the standard product metric on $S^3\times S^2$, and let $k\in \dZ$ be even. Then there exists a diffeomorphism $\phi:S^3\times S^2\to S^3\times S^2$ and a family of metrics $(S^3\times S^2,g_t)$ such that
	\begin{enumerate}
		\item $\Ric_t>0$ for all $t\in [0,1]$
		\item The $S^1$-action $\cdot_{(1,k)}$ on $S^3\times S^2$ is an isometric action for all $g_t$.
		\item $g_1 = \phi^*g_0$ with $\phi\big(\theta\cdot_{(1,k)}(s_1,s_2)\big)= \theta\cdot_{(1,0)}\phi(s_1,s_2)$ for any $s_1\in S^3$ and $s_2\in S^2$.
	\end{enumerate}
\end{theorem}

\begin{remark}\label{remark:isotopic_class}
	The diffeomorphism $\phi:S^3\times S^2\to S^3\times S^2$ in Theorem \ref{t:equivariant_mapping_class_S3xS2} above can be taken isotopic to a diffeomorphism $\psi: S^3\times S^2 \to S^3\times S^2$ with the following special structure
	\begin{equation}
		(s_1,s_2) \mapsto (s_1, \psi_{s_1}(s_2)) \, , \quad \psi_{s_1}\in O(3)\, , \, \, \forall\, s_1\in S^3\, .
	\end{equation}
 We refer the reader to Section \ref{subsection:isotopy class} for details and remark here that any such diffeomorphism can be trivially extended to a diffeomorphism $\overline{\psi}:S^3\times D^3\to S^3\times D^3$.
\end{remark}

\vspace{.2cm}

The above tells us that we can find an $S^1$-invariant family of metrics with positive Ricci curvature which (from an isometric point of view) start and end at the classical $S^3\times S^2$, however, the beginning and ending $S^1$ actions are quite distinct.  Our main use of the above will be to build the following neck region, which will be used to alter $M_j$ to $\hat M_j$:

\begin{proposition}[Step 2: Twisting the Action]\label{p:step2}
	Let $\epsilon,\hat \epsilon,\delta>0$ with $k\in \dZ$ even.  Then there exist $\hat \delta(\epsilon,\hat \epsilon,\delta,k)>0$ and $R(\epsilon,\hat \epsilon,\delta,k)>1$ and a metric space $X$ with an isometric and free $S^1$ action such that
	\begin{enumerate}
		\item $X$ is smooth away from a single three sphere $S^3_\delta\times \{p\}\subset X$ with $\Ric _X\geq 0$.
		\item There exists $B_{10^{-3}}(p)\subseteq U\subseteq B_{10^{-1}}(p)\subseteq X$ which is isometric to $S^3_\delta\times B_{10^{-2}}(0)\subseteq S^3_\delta\times C(S^2_{1-\epsilon})$ , and under this isometry the $S^1$ action on $U$ identifies with the $(1,k)$ action. 
		\item There exists $B_{10^{-1}R}(p)\subseteq \hat U\subseteq B_{10R}(p)\subseteq X$ s.t. $X\setminus \hat U$ is isometric to $S^3_{\hat \delta R}\times A_{R,\infty}(0)\subseteq S^3_{\hat \delta R}\times C(S^2_{1-\hat \epsilon})$, and under this isometry the $S^1$ action on $X\setminus \hat U$ identifies with the $(1,0)$ action. 
	\end{enumerate}
\end{proposition}

Given Theorem \ref{t:equivariant_mapping_class_S3xS2}, the proof of Proposition \ref{p:step2} is completely analogous to the proof of the corresponding \cite[Proposition 3.3]{BrueNaberSemolaMilnor1} given \cite[Theorem 3.2]{BrueNaberSemolaMilnor1}, see \cite[Section 7]{BrueNaberSemolaMilnor1} for the details. Hence it will be omitted.\\

{\bf Constructing $\hat M_j$:} Before moving on to Step 3, let us see how the above will be used as part of our induction process. Thus let us assume we have constructed $M_j$ as in (I1)-(I3) with sphere radius $\delta_j$.  Recall by (I3) that outside of a compact subset we have that $M_j$ is isometric to $S^3_{\delta_j r_j}\times C(S^2_{1-\epsilon_j})$, and the action of $\gamma_j$ rotates the $S^2_{1-\epsilon_j}$ cross section by $2\pi/k_j$ and Hopf rotates the $S^3_{\delta_j r_j}$ factor by $2\pi/|\gamma_j|$.  Observe that if we consider the $(1,k_{\le j-1})$-action on $S^3_{\delta_j r_j}\times C(S^2_{1-\epsilon_j})$, then $\Gamma_j\subseteq S^1$ can be viewed as a subaction.\\

Now with any $\hat\epsilon_j>0$, the precise constant will be chosen later, we have for $R_j=R_j(\epsilon_j,\hat\epsilon_j,\delta_j,k_{\le j-1})$ and $\hat\delta_j=\hat\delta_j(\epsilon_j,\hat\epsilon_j,\delta_j,k_{\le j-1})$ the existence of $X_j$ as in Proposition \ref{p:step2}, where we chose $k=k_{\le j-1}=k_0\cdot k_1\cdots k_{j-1}$ in the application of the Proposition.

  We can rescale $X_j\to r_j X_j$ by $r_j$ so that it is isometric to $S^3_{\delta_j r_j}\times C(S^2_{1-\epsilon_j})$ on a region $U$ containing $B_{r_j}(p)$, and it is isometric to $S^3_{\hat \delta_{j}R_j r_j}\times C(S^2_{1-\hat \epsilon_{j}})$ on a region $X_j\setminus \hat U$ containing the annulus $A_{R_j r_j,\infty}(p)$.\\  Further, there is a free isometric $S^1$ action on $X_j$ which looks like the $(1,k_{\le j-1})$ action on $U$ and the $(1,0)$ action on $X_j\setminus \hat U$.  In particular, by condition (2) in Proposition \ref{p:step2} and the inductive assumption (I3) there is an induced $\Gamma_j$ action on $X_j$ and an open annulus of $U\subseteq X_j$ which is equivariantly isometric to an open annulus in $M_j\setminus U_j$.\\

We can thus glue $X_j$ to $M_j$ in order to produce the space $\hat M_j$.  The space $\hat M_j$ is now isometric to $S^3_{\hat \delta_{j}R_j r_j}\times C(S^2_{1-\hat \epsilon_{j}})$ outside of some compact set $V_j\subseteq \hat M_j$, and the $\Gamma_j$ action is a pure Hopf rotation on the $S^3_{\hat \delta_{j}R_j r_j}$ factor on $\hat M_j\setminus V_j$.\\

\vspace{.3cm}

\subsubsection{\bf Step 3: Gluing Construction}\label{ss:step 3 outline}

The third step of the construction involves extending the action of $\Gamma_j$ to an action of $\Gamma_{j+1}$ in order to move from the manifold $M_j$ to the next step of the induction $M_{j+1}$.  This will occur by taking $k_{j+1}$ copies of the twisted space $\hat M_j$, constructed in the second step, and gluing them into a model space $\cB_{j+1}\approx \cB(\epsilon_{j+1},\delta_{j+1})$ constructed in the first step.\\

Recall that a model space $\cB(\epsilon,\delta)$ is isometric to an annulus in $S^3_{\delta}\times C(S^2_{1-\epsilon})$ outside of a compact set, and recall that the induction manifolds $\hat M$ are isometric to annuli in $S^3_{\delta}\times C(S^2_{1-\hat\epsilon})$ outside of a compact set.  We will therefore outline our gluing constructions purely in terms of annuli, which is where the gluing will take place.  If we can accomplish this with the correct behaviors, we can then glue our model space $\cB_{j+1}$ and inductive manifolds $\hat M_j$ directly into our glued space and finish the inductive construction of $M_{j+1}$.\\

Let us first outline the gluing strategy without worrying about smoothness or Ricci curvature.  We will end with Proposition \ref{p:step3}, which will state the end construction in a smooth Ricci preserving manner. So let $\cA' \equiv S^3_{\delta}\times C(S^2_{1-\epsilon})$ and let $\hat \cA = S^3_\delta\times B_1(0)\subseteq S^3_\delta\times C(S^2_{1-\hat\epsilon})$ with $\Gamma\leq S^1$ a finite group generated by a single element $\gamma$ whose order is divisible by $k$.  Let $\hat\Gamma$ be the group generated by $\hat\gamma\equiv\gamma^k$.  Consider the action of $\Gamma$ on $\cA'$ induced by the $(1,|\gamma|/k)$-action.  Thus $\gamma$ rotates the $S^2_{1-\epsilon}$ factor by $2\pi/k$ and Hopf rotates the $S^3_\delta$ factor by $2\pi/|\gamma|$.  Let us also consider the action of $\hat\Gamma$ on $\hat \cA$ obtained by just rotating the $S^3_\delta$ factor by $2\pi/|\hat\gamma|$ . \\

Consider $k$ copies of the annulus {$\hat\cA^a \equiv \hat\cA \times \{a\} $} with $a=0,\ldots,k-1$, and note that $\partial\hat\cA^a = S^3_\delta\times S^2_{1-\hat\epsilon}$ isometrically.  Our goal is to glue in these $k$ copies into $\cA'$ such that there is an induced $\Gamma$ action on the glued space.  We will want that $\hat\Gamma$ restricts to the usual actions on both $\cA'$ and the glued copies of $\hat \cA$.  To be more precise let $x\in C(S^2_{1-\epsilon})$ be a point whose distance from the origin is $10^2k$.  Let $x^a\in C(S^2_{1-\epsilon})$ with $a=0,\ldots, k-1$ be the $k$ points  obtained by rotating $x^0=x$ by $2\pi a/k$. \\

Consider each of the domains $S^3_\delta\times B_{1}(x^a)\subseteq \cA'$, and note that their boundaries are diffeomorphic (and nearly isometric) to $S^3_{\delta}\times S^2_{1}$.  Note that the $\hat \Gamma$ action restricts to actions on each of these domains, while the $\Gamma$ action simply restricts to an isometry between potentially different pairs of domains.  We will want to glue $\hat \cA^0, \ldots, \hat \cA^{k-1}$ into the space

\begin{align}
	\cA'\setminus \Big(\bigcup_a S^3_{\delta}\times B_{1}(x^a)\Big)\, .
\end{align}

In order to perform the gluing we need to define the gluing diffeomorphisms
\begin{align}
	\varphi^a: \partial \hat \cA^a\to S^3_{\delta}\times \partial B_{1}(x^a)\, .
\end{align} 

Recalling that $\partial \hat \cA^0=S^3_{\delta}\times S^2_{1}$ and $S^3_{\delta}\times\partial B_{1}(x)$ is nearly isometric to $S^3_\delta\times S^2_{1}$, let us first choose an almost isometry $\varphi^0:\partial \hat \cA^0\to S^3_\delta\times\partial B_{ 1}(x)$ which is the identity on the first sphere factor.  In particular, it follows that $\varphi^0$ commutes with the natural $\hat\Gamma$ actions on each of these spaces.  Let us then define $\varphi^a: \partial \hat \cA^a\to S^3_\delta\times\partial B_{1}(x^a)$ by

\begin{align}
	\varphi^a(y,a) = \gamma^a\cdot \varphi^0(y,0)\, , \quad y\in \hat \cA\, ,
\end{align}
for $a=0, \ldots, k-1$.
Note that we could naturally extend the above maps for any $a\in \dZ$.  However, we would have that $\varphi^{k}:\partial \hat \cA^0\to S^3_\delta\times\partial B_{1}(x^0)$ would not be the same mapping as $\varphi^0$.  Indeed, we see that $\varphi^{k}= \gamma^{k}\cdot \varphi^0 = \hat\gamma\cdot \varphi^0$.  To understand the implications of this consider the glued space
\begin{align}
	\tilde \cA \equiv \Big(\cA'\setminus \bigcup_a S^3_\delta\times B_{1}(x^a)\Big)\bigcup_{\varphi^a} \hat \cA^a\, ,
\end{align}
where we have plucked out the $k$ domains $S^3_\delta\times B_{1}(x^a)$ and plugged in the new annular regions $\hat \cA^a$.  The new space $\tilde \cA$ is still isometrically of the form $S^3_\delta\times C(S^2_{1-\epsilon})$ near the origin and infinity.  The effect of the gluing maps is that the $\hat\Gamma$ action on $\hat\cA$ extends to a $\Gamma$ action on $\tilde \cA$.  To understand this action, we need to describe the action of $\gamma$ on $\bigcup_{a} \hat \cA^a$. The latter is given by
\begin{equation}
	\begin{split}
	&\gamma\cdot (y,a) = (y,a+1)\, , \quad a=0, \ldots , k-2\, 
	\\
	&\gamma\cdot (y,k-1) = (\hat \gamma\cdot y,0)
	\end{split}
\end{equation}
for every $y\in \hat \cA$.  In particular, the action of $\hat\Gamma$ restricts to the expected action on each piece of the gluing.\\

The main Proposition of this step is to show that, up to some altering of constants, the above construction can be smoothed to preserve nonnegative Ricci curvature:

\begin{proposition}[Step 3:  Action Extension]\label{p:step3}
	Let $\epsilon,\epsilon',\delta>0$ with $0<\epsilon-\epsilon'\le \frac{1}{10^2 }\epsilon$, and let $\hat\Gamma\leq \dQ/\dZ\subseteq S^1$ be a finite subgroup with $\Gamma = \langle \gamma, \hat\Gamma\rangle$ such that $\hat\gamma\equiv \gamma^k$ is the generator of $\hat\Gamma$.  Then for $\hat\epsilon\leq \hat\epsilon(\epsilon,\epsilon')$ there exists a pointed space $(\tilde \cA,p)$, isometric to a smoooth Riemannian manifold with $\Ric\ge 0$ away from $k+1$ three spheres, with an isometric and free action by $\Gamma$ such that
	\begin{enumerate}
		\item There exists a $\Gamma$-invariant set $B_{10^{-1}}(p)\subseteq U'\subseteq B_{10}(p)$ which is isometric to $S^3_\delta\times B_{1}(0)\subseteq S^3_\delta\times C(S^2_{1-\epsilon'})$ and such that $\Gamma$ is induced by the $(1,|\gamma|/k)$-action on $S^3_\delta\times S^2_{1-\epsilon'}$ ,
		\item There exists a $\Gamma$-invariant set $B_{10^3 k}(p)\subseteq U\subseteq B_{10^5 k}(p)$ such that $\tilde \cA\setminus U$ is isometric to $S^3_\delta\times A_{10^4k,\infty}(0)\subseteq S^3_\delta\times C(S^2_{1-\epsilon})$ and such that $\Gamma$ is induced by the $(1,|\gamma|/k)$-action on $S^3_\delta\times S^2_{1-\epsilon}$ 
		\item There exist $\hat\Gamma$-invariant sets $S^3_\delta\times B_{2^{-1}}(x^a)\subseteq V^a\subseteq S^3_\delta\times B_{2}(x^a)$ with $d(S^3_\delta\times \{x^a\},S^3_\delta\times \{p\})=10^2 k$ which are isometric to $S^3_\delta\times B_{1}(0)\subseteq S^3_\delta\times C(S^2_{1-\hat\epsilon})$ and such that $\hat\Gamma$ is induced by the $(1,0)$-action on $S^3_\delta\times S^2_{1-\hat\epsilon}$ .
	\end{enumerate}
\end{proposition}

Proposition \ref{p:step3} can be proved by following verbatim the proof of \cite[Proposition 3.4]{BrueNaberSemolaMilnor1}, with very minor changes due to the fact that we are working with $S^3 \times S^2$ instead of $S^3\times S^3$. Hence its proof will be omitted.

\begin{remark}
	It is important to observe that $\hat\epsilon(\epsilon,\epsilon')$ depends on the choices of $\epsilon$ and $\epsilon>\epsilon'$, however, it does not depend on the choice of $\delta$.  
\end{remark}

{\bf Constructing $M_{j+1}$:}  Let us now apply Proposition \ref{p:step3} in order to finish the construction of $M_{j+1}$.  Let us take in the above $\Gamma=\Gamma_{j+1}$ and $\hat\Gamma=\Gamma_j$, and let us choose $\epsilon=\epsilon_{j+1}$ with $\epsilon'=\epsilon_{j+1}\cdot \frac{99}{100}$.  Recall that the construction of $\hat M_j$ in Section 2 depended on a choice of $\hat\epsilon_j$, which had not yet been fixed.  Let us now use Proposition \ref{p:step3} in order to choose $\hat\epsilon_j = \hat\epsilon_j(\epsilon_{j+1})$.  From this we now have from Proposition \ref{p:step2} a well-defined $R_j$ and $\hat\delta_j$.  Finally let us now choose $\delta=\hat\delta_j$ in the application of Proposition \ref{p:step2}, so that we have built the space $\tilde \cA_j$.  After rescaling $\tilde \cA_j\to (R_jr_j) \tilde \cA_j$ by $R_jr_j$ observe that there exists $U\subseteq \tilde \cA_j$ which is isometric to $S^3_{\hat\delta_j R_j r_j}\times B_{R_jr_j}(0)\subseteq S^3_{\hat\delta_j R_j r_j}\times C(S^2_{1-\epsilon'})$, and also observe that the domains $V^a$ are isometric to $ S^3_{\hat\delta_j R_j r_j}\times B_{R_jr_j}(0)\subseteq S^3_{\hat\delta_j R_j r_j}\times C(S^2_{1-\hat\epsilon_{j}})$.\\

Finally, let us consider the base model $\cB_{j+1}=\cB(\epsilon', \hat\delta_j R_j r_j)$ from Proposition \ref{p:step1}.  We see we can glue it isometrically into $U\subseteq \tilde \cA_j$.  Additionally, we can isometrically glue $\hat M_j$ into each $V^a\subseteq \tilde \cA_j$.  The resulting space is $M_{j+1}$.  If we define $p_{j+1}=p_j^0$ to be the basepoint of the copy of $M_j$ glued into $V^0$, then we can define $r_{j+1}\equiv d(p_{j+1},\gamma_{j+1}\cdot p_{j+1})$ and $\delta_{j+1}$ through the formula $\delta_{j+1} r_{j+1}\equiv \hat\delta_j R_j r_j$.  This completes the induction step of the construction.  In particular, we have proved Theorem \ref{t:milnor6d} up to the proof of Theorem \ref{t:equivariant_mapping_class_S3xS2}. $\qed$\\

\vspace{.5cm}

\section{The topology of the universal cover}\label{sec:univdiffeo}

The goal of this section is to prove that the universal cover $\tilde M$ of the example constructed in Theorem~\ref{t:milnor6d} is diffeomorphic to $S^3 \times \mathbb{R}^3$ when the twisting diffeomorphisms $\phi_{2k}: S^3 \times S^2 \to S^3 \times S^2$ are chosen as in Remark~\ref{remark:isotopic_class} at each step of the inductive construction.
In order to do so, we will build a Morse-Bott function with no critical points outside from a central $S^3$.

\begin{proposition}\label{prop:Morse}
	There exists a proper smooth function $f: \tilde M\to [0,\infty)$ such that
	\begin{itemize}
		\item[(1)] The set $\{f < 1\}$ is diffeomorphic to $S^3 \times \dR^3$.

		\item[(2)] $f$ does not have critical points in $\{f>0\}$.
	\end{itemize}
\end{proposition}
Given $f:\tilde M \to \dR$ as in Proposition \ref{prop:Morse}, it is standard to check that $\tilde M \approx S^3 \times \dR^3$. Indeed, $f$ is proper and has no critical points in $\{1\le f <\infty\}$. So, by Morse theory, $\tilde M$ is diffeomorphic to $\{f < 1\} \approx S^3 \times \dR^3$. As will be clear from the construction, each of the manifolds $M_j$ in the inductive proof is also diffeomorphic to $S^3\times \dR^3$.

\begin{remark}
Notice that a single gluing of $S^3\times D^3$ with $S^3\times \dR^3\setminus S^3\times D^3$ by a gluing diffeomorphism $\phi:S^3\times S^2\to S^3\times S^2$ is diffeomorphic to $S^3\times \dR^3$ independently of the isotopy class of the diffeomorphism. This claim can be established with a much easier variant of the construction that we discuss below.
\end{remark}

\begin{remark}
Independently of the construction of the Morse-Bott function on $\tilde{M}$, which yields the diffeomorphism with $S^3\times \dR^3$, it would be possible to compute the homology of each of the manifolds $M_j$ and of $\tilde{M}$ by iterated use of the Mayer-Vietoris sequence. As in the construction of the Morse-Bott function, the explicit form of the gluing diffeomorphisms plays a key role in showing that the homology coincides with that of $S^3\times \dR^3$.
\end{remark}

\begin{remark}
By relying on the fact that each of the gluing diffeomorphisms $\phi_k:S^3\times S^2\to S^3\times S^2$ extends to a diffeomorphism $\overline{\phi}_k:S^3\times D^3\to S^3\times D^3$ it would be possible to check that each of the manifolds $M_j$ is diffeomorphic to $S^3\times\dR^3$ in a more classical way. Our argument is designed so that it can be used to conclude that the limit $\tilde{M}$ is diffeomorphic to $S^3\times \dR^3$ as well.
\end{remark}

Let us begin with an outline of the proof of Proposition \ref{prop:Morse}, and then the remaining parts of this section are devoted to the details of the proof.\\ 

The Morse-Bott function $f$ will be built inductively on each $\hat M_j$, where recall $\hat M_j$ is $M_j$ twisted at infinity as in Step 2 of the construction. The inductive construction will be such that $f_{j+1}:\hat{M}_{j+1}\to [0,\infty)$ coincides with $f_j$ on $\Phi_j(\hat U_j)\subset \hat{M}_{j+1}$ for any $j\in\dN$, see the introductory discussion of Section \ref{s:outline_milnor:inductive} for the relevant notation. Hence the definition of the global function $f:\tilde{M}\to [0,\infty)$ will be straightforward once the inductive construction has been completed.\\

Recall that $\hat M_j \setminus \hat U_j$ is diffeomorphic to the complement of a ball in $S^3 \times \dR^3$. The Morse-Bott function $f_j$ in $\hat M_j$ at infinity will by assumption always look like the distance squared from the origin.  The base step of the induction will be elementary: In $\hat M_1 = S^3 \times \dR^3$ we define $f_1$ as the distance squared from the origin of $\dR^3$, independent of the $S^3$ factor.

The inductive step proceeds roughly as follows. Let $\{x^a\}_{a=0,\ldots,k_{j+1}-1}\subset \dR^3$ be the centers of the disks $D^3_a$ that we remove from $\dR^3$ to glue in $k_{j+1}$-copies of $\hat M_j$. In $S^3 \times (\dR^3 \setminus \bigcup_a D_a^3)$ we define the function $f_{j+1}$ as the distance squared from $x^0$, independent from the $S^3$ factor. The function $f_{j+1}$ matches with $f_j$ in a neighborhood of the gluing region. Thus when we glue in the first copy of $\hat M_j$ along an annular region near $\partial(S^3 \times D_0^3)$ we can extend $f_{j+1}$ using $f_j$. At this point, it remains to glue in the other $k_{j+1}-1$ copies of $\hat M_j$ and extend $f_{j+1}$ to a function {\it without} critical points in each of these copies of $\hat M_j$. In order to achieve this extension, we rely on a second induction procedure, which is the most delicate part of our argument. Here is where it is helpful to know that the gluing diffeomorphisms $\phi_{2k} : S^3 \times S^2 \to S^3 \times S^2$ are isotopic to diffeomorphisms with the special structure
\begin{equation}\label{eq:z}
	\psi(s_1,s_2) = (s_1, \psi_{s_1}(s_2))\, , \quad  \, \, s_1, \in S^3\,\, s_2\in S^2 \,,
\end{equation} 
where $\psi_{s_1}$ is an orthogonal transformation in $O(3)$ for each $s_1\in S^3$ and the dependence on $s_1$ is smooth.

Roughly, the idea is the following.
Since we are only interested in the diffeomorphism class of $\tilde{M}$, we can replace all the gluing diffeomorphisms in the construction with isotopic diffeomorphisms.
Thus without loss we assume that the gluing of $\hat M_j$ along $\partial (S^3 \times D_a^3)$ are induced by a diffeomorphism $\psi$ with the structure in \eqref{eq:z}. We then need to extend the pull-back of $f_{j+1}|_{\partial (S^3 \times  D_a^3)}$ through $\psi$ to a function without critical points in $\hat M_j$. 
Notice that $f_{j+1}$ looks like a nontrivial affine function of the $\dR^3$ factor in a neighborhood of the boundary $\partial (S^3 \times D_a^3)$, $a\neq 0$. Up to a small perturbation that does not introduce any critical point, we can, and will, assume that it is affine on each $\dR^3$ factor.\\
From \eqref{eq:z} we then deduce that the pullback of $f_{j+1}$ through $\psi$ looks like an affine function in the $\dR^3$ factor of $S^3 \times \dR^3$, outside of a compact set in $\hat M_j$. Notice that the pull-back is no longer independent of the $S^3$ factor, though it is affine on $\dR^3$ for each fixed point in $S^3$.  By employing yet another induction argument, we will see that a function with this property can be extended to a function without critical points.  This will be carried out in subsection \ref{subsec:extension}.  In order to complete the proof, it remains only to slightly perturb the function $f_{j+1}$ outside of a compact set and without introducing further critical points so that it again coincides with the distance squared from the origin at infinity in $\hat M_{j+1}$. \\

\subsection{Extension without critical points}\label{subsec:extension}

Recall that, for each $\hat M_j$ there are open sets $\hat U_j \subset \hat M_j$ such that $\hat M_j \setminus \hat U_j \approx S^3 \times A_{r_j,\infty}(0)$ with $r_j\uparrow \infty$.
To build $\hat M_{j+1}$, we start from $S^3 \times (\dR^3 \setminus \bigcup_a B_{r_j}(x^a))$, where $\{x^a\}_{a=0,\ldots,k_{j+1}-1}\subset \dR^3$ are chosen such that the annular regions $A_{r_j,2r_j}(x^a) \subset \dR^3$ are disjoint. The diffeomorphism class of the construction is independent of the positions of their centers, however for the sake of clarity we assume that $\{x^a\}_{a=0, \ldots, k_{j+1}-1} \subset \partial B_{r_{j+1}/10}(0)$ is the orbit of a rotation by angle $2\pi/k_{j+1}$, and $r_{j+1} \gg r_j \gg 1$.

We then glue in $k_{j+1}$ copies of $\hat M_j$ along the neck regions $X_j \approx S^3 \times A_{r_j,2r_j}(x^a)$, through diffeomorphisms $\psi_a : X_j \to S^3 \times A_{r_j,2r_j}(x^a)$ obtained by radially extending $\psi$ as
\begin{equation}\label{eq:psi_a}
	S^3 \times A_{r_j,2r_j}(0)\ni (s_1,r,s_2) \mapsto (s_1,r,\psi_{s_1}(s_2))\in S^3 \times A_{r_j,2r_j}(0)
\end{equation}
and composing with a suitable isometry of $S^3\times \dR^3$ which maps $S^3 \times A_{r_j,2r_j}(0)$ to $S^3 \times A_{r_j,2r_j}(x^a)$.

\begin{lemma}\label{lemma:extension1}
	Fix $j\in \mathbb{N}$, $a\in \{0,\ldots, k_{j+1}-1\}$, and $x^a$ as above. Let $u: S^3 \times A_{r_j,2r_j}(x^a) \to \dR$ be such that $u(x, \cdot): A_{r_j,2r_j}(x^a) \to \dR$ is the restriction of a non-constant affine function of $\dR^3$. Then, there exists a smooth function $u_j^a : \hat M_j \to \dR$ such that
	\begin{itemize}
		\item[(1)] $u_j^a = u \circ \psi_a$ in $\hat M_j \setminus \hat U_j$.
		
		\item[(2)] $u_j^a$ does not have critical points in $\hat U_j$.
		
		\item[(3)] $\inf_{\hat U_j} u_j^a \ge \inf_{S^3 \times A_{1,2}(x^a)} u$.
	\end{itemize}
\end{lemma}

\begin{proof}
	We proceed by induction in $j\ge 1$. The base case $j=1$ is trivial.
	
For the inductive step let us begin by remarking that as a consequence of \eqref{eq:z}, $\psi$ maps $\{x\}\times S^2$ to itself isometrically for each $x\in S^3$. In view of \eqref{eq:psi_a}, we deduce that
	\begin{equation}
		u\circ \psi_a : S^3 \times A_{r_j,2r_j}(0) \to \dR\, ,
	\end{equation}
	retains the property that $u\circ \psi_a(x,\cdot) : A_{r_j,2r_j}(0)\subseteq \dR^3 \to \dR$ is the restriction of a non-constant affine function. In particular, we can uniquely extend it to  $\tilde u_j: S^3 \times \dR^3 \to \dR$ such that $\tilde u_j(x,\cdot): \dR^3 \to \dR$ is affine and non-constant.  Consequently $\tilde u_j$ has no critical points, as its derivative in the $\dR^3$ directions are nonzero. Notice that by construction $\inf_{S^3 \times B_2(0)} \tilde u_j \ge \inf_{S^3\times A_{1,2}(x^a)}u$.
	To define a Morse function on $\hat M_j$, we first have to pluck-out $k_{j}$ copies of $S^3 \times D^3$ from $S^3\times \dR^3$ and glue in $k_{j}$ copies of $\hat M_{j-1}$ along the annular regions.  We then need to smoothly extend $\tilde u_j$ in each copy of $\hat U_{j-1}$. However, $\tilde u_j$ restricted to the gluing regions satisfies the assumptions of Lemma \ref{lemma:extension1}, hence we can apply our inductive hypothesis to conclude the result.
\end{proof}

\vspace{.3cm}

\subsection{Inductive construction of the Morse function}

The following inductive lemma provides the Morse-Bott functions $f_j :  M_j\to \dR$.

\begin{lemma}\label{lemma:inductionMorse}
	For every $j\ge 1$, there exists a proper smooth function $f_j: \hat M_j \to [0,\infty)$ such that
	\begin{itemize}
		\item[(i)] $f_j=f_{j+1}$ in $ \Phi_j(\hat U_j) \subset \hat M_{j+1}$.

		\item[(ii)] $f_j(y)=r^2(y)$ on $\hat M_j\setminus \hat U_j \approx S^3 \times A_{r_j,\infty}(0)$, where $r(y)$ coincides with the distance to $S^3 \times \{0\}$. 		
		\item[(iii)] $f_1: \hat M_1 \approx S^3 \times \dR^3 \to [0,\infty)$ is defined by $f_1(s,x) = |x|^2=r^2(s,x)$.
		
		\item[(iv)] $f_j$ does not have critical points in $\{f_j>0\}$.
		\item[(v)] we have that $\inf_{\hat U_{j+1}\setminus \hat U_{j}} f_{j+1} \to \infty$ as $j\to \infty$ .
	\end{itemize}
\end{lemma}
\vspace{.3cm}

By condition (i), we obtain a naturally defined Morse-Bott function $f:\tilde M\to [0,\infty)$ by letting $f:=f_j$ on $\hat U_j\subset \tilde{M}$, with the obvious identifications. By (v) we have that $f$ is proper. By (iii), the sublevel set $\{f<1\}$ is diffeomorphic to $S^3\times \dR^3$. By (iv), $f$ does not have any critical points on $\{f>0\}$.
Thus the proof of Proposition \ref{prop:Morse} will be complete once Lemma \ref{lemma:inductionMorse} is proved.

\begin{proof}[Proof of Lemma \ref{lemma:inductionMorse}]
For the base step, we let $f_1: M_1 \equiv S^3 \times \dR^3 \to [0,\infty)$ be defined by $f_1(s,x) = |x|^2$. In particular, condition (iii) is satisfied.
\smallskip

For the inductive step, we start from $S^3 \times (\dR^3 \setminus \bigcup_a B_{r_j}(x^a))$, where $\{x^a\}_{a=0,\ldots,k_{j+1}-1}\subset \partial B_{r_{j+1}/10}(0)$ is invariant under a rotation by angle $2\pi/k_{j+1}$, and $r_{j+1} \gg r_{j} \gg 1$ is big enough so that the annular regions $A_{r_j,2r_j}(x^a) \subset \dR^3$ are disjoint with $0\notin A_{r_j,2r_j}(x^a)$.
\\ 
We claim that there exists a smooth function without critical points $\eta:\dR^3 \setminus \bigcup_a B_{r_j}(x^a)\to [0,\infty)$ such that 
\begin{itemize}
\item[a)] $\eta(x)=|x-x^0|^2$ on $A_{r_j,2r_j}(x^0)$;
\item[b)] $\eta$ is affine on each annulus $A_{r_j,2r_j}(x^a)$ for $a=1,\dots,k_{j+1}-1$;
\item[c)] $\eta(x)=|x|^2$ on $\dR^3\setminus B_{100r_j}(0)$;
\item[d)] $\eta(x)\ge \frac{1}{2}|x-x^0|^2$ on $\dR^3 \setminus \bigcup_a B_{r_j}(x^a)$.
\end{itemize}
The existence of a function $\eta$ with the properties above is completely elementary, therefore we omit the proof.
\smallskip

In order to define $f_{j+1}: \hat M_{j+1}\to [0,\infty)$, we set can set it to coincide with the function $\eta$ on $\hat M_{j+1}\setminus (\bigcup_{a=0}^{k_j-1}\hat{M}_j^a)$. Up to choosing properly the parameters $r_j\gg r_{j-1}$, the functions $\eta$ and $f_j$ coincide in the gluing region by a), hence $f_{j+1}$  extends to $\hat M_{j+1}\setminus (\bigcup_{a=1}^{k_j-1}\hat{M}_j^a)$. In this way, taking into account also c) above, it is clear that conditions (i), (ii) and (iii) are met by $f_{j+1}$. \\
The extension to $\hat M_{j+1}$ can then be achieved with the help of Lemma \ref{lemma:extension1}, thanks to condition b) above.\\ 
As the extensions from Lemma \ref{lemma:extension1} do not introduce any critical points, it is also clear from the construction that $f_{j+1}$ does not have any critical points in $\{f_{j+1}>0\}$.\\
The validity of (v) then follows from condition d) above and from ii).

\end{proof}

\vspace{.5cm}

\section{Further properties of the counterexample}\label{sec:furtherproperties}

In this section we discuss some further geometric properties of the counterexamples to the Milnor conjecture constructed in this paper, and of their universal covers.  In particular we prove Theorem \ref{t:milnor6d}.1-\ref{t:milnor6d}.4 and additionally study the tangent cones at infinity of our examples.  We remark that all the statements adapt, mutatis mutandis, to the $7$-dimensional counterexamples constructed in our earlier work \cite{BrueNaberSemolaMilnor1}.

\subsection{Curvature decay}

The goal of this subsection is to prove the following, which is equivalent to Theorem \ref{t:milnor6d}.1:

\begin{proposition}\label{prop:curvdecay}
For any $\eta>0$, a complete manifold $(M^6,g,p,\Gamma)$ as in Theorem~\ref{t:milnor6d} can be constructed so that it satisfies
\begin{equation}\label{eq:eta_curvature}
|Rm|(q)\le \frac{C}{d(p,q)^{2-\eta}}\, , \quad \text{for every $q\in M$}\, ,
\end{equation}
for some constant $C>0$.
\end{proposition}

\begin{remark}
Of course, no curvature decay can be expected for the universal covering of a non-flat manifold with an infinite fundamental group. However, a straightforward corollary of Proposition~\ref{prop:curvdecay} above is that the universal covers of the counterexamples to the Milnor conjecture constructed in this paper can be constructed in such a way that they have bounded sectional curvature.
\end{remark}

\begin{remark}
A slight modification of the proof of Proposition~\ref{prop:curvdecay} would show that a complete manifold $(M^6,g,p,\Gamma)$ as in Theorem~\ref{t:milnor6d} can be constructed so that for any $\eta>0$ it satisfies 
\begin{equation}\label{eq:eta_curvaturebis}
|Rm|(q)\le \frac{C(\eta)}{d(p,q)^{2-\eta}}\, , \quad \text{for every $q\in M$}\, ,
\end{equation}
for some constant $C(\eta)>0$.
\end{remark}

	The proof of Proposition \ref{prop:curvdecay} is carried out by induction with respect to the exhaustion $(U_j,p_j)$, where $U_j \subset M_j$. Let $\pi : \tilde M^6 \to M^6$ be the covering map, and $\pi_j: \tilde M^6 \to \tilde M^6/\Gamma_j$ the projection to the intermediate quotient. See Section~\ref{ss:inductiveconstruction} for the relevant notation.
	It is clear that $\pi(U_j) = \pi_j(U_j)$ provides an exhaustion of $M^6$. 
	
	We prove the following claim by induction.
	
	\medskip

	{\bf Claim:} Let $j\in \mathbb{N}$. 
	We can build $(M_j, p_j, \Gamma_j)$ so that
	\begin{itemize}
		\item[(i)] For every $q\in \pi_j(U_j)$ all sectional curvatures are bounded by 
		\begin{equation}\label{claim1}
			|Rm|(q) \le \frac{10^{100}}{d(q,p_j)^{2-\eta}}\, .
		\end{equation}

	    \item[(ii)] For every $q\in \pi_j(A_{k_jr_{j},100k_jr_{j}}(p_j))$ all sectional curvatures are bounded by 
	    \begin{equation}\label{claim2}
	    	|Rm|(q) \le \frac{1}{d(q,p_j)^{2-\eta}}\, .
	    \end{equation}
	\end{itemize}

	\medskip

	It is obvious that we can build $U_1$ so that the Claim is satisfied. We now show that if it is satisfied for $j\in \mathbb{N}$, then it is satisfied also for $j+1$.
	
	As a first step, we build $\hat M_j$ so that $\pi_j(\hat M_j)$ satisfies suitable curvature bounds up to scale $r_{j+1} \gg r_j$. More precisely we have:
	
	\begin{lemma}\label{lemma:curvature}
		Fix $j\in \mathbb{N}$ and $\mu>0$. Assume that \eqref{claim1} and \eqref{claim2} hold true. Then, we can build $\hat M_{j+1}$ such that:
		\begin{itemize}
			\item[(a)] For every $q\in \pi_j(B_{100r_{j+1}}(p_j))$ all sectional curvatures are bounded by
			\begin{equation}\label{claim1.1}
				|Rm|(q) \le \frac{10^{100}}{d(q,p_j)^{2-\eta}}\, .
			\end{equation}
		    
		    \item[(b)] For every $q\in \pi_j(A_{r_{j+1}, 10r_{j+1}}(p_j))$ all sectional curvatures are bounded by
		    \begin{equation}\label{claim1.2}
		    	|Rm|(q) \le \frac{\mu}{r_{j+1}^{2-\eta}}\, .
		    \end{equation}
		\end{itemize}
	\end{lemma}

	Before proving Lemma~\ref{lemma:curvature}, let us first check how to conclude the proof of Proposition \ref{prop:curvdecay} assuming its validity.\\
	
	By construction, $M_{j+1}$ is obtained by gluing $k_{j+1}$ distinct copies of $\hat M_j$ to a slight perturbation of $S^3_{\delta_{j+1}r_{j+1}} \times C(S^2_{1-\epsilon_{j+1}})$ after removing $k_{j+1}$-copies of $S^3 \times D^3$, as explained in section \ref{s:outline_milnor:inductive}. The errors introduced by this perturbation are uniformly controlled, hence we will neglect them for the sake of this argument.
	
	A lift $\tilde q\in M_{j+1}$ of any $q\in \pi(U_{j+1})$ either belongs to a copy of $\hat M_j$ or to $S^3_{\delta_{j+1}r_{j+1}} \times C(S^2_{1-\epsilon_{j+1}})$. In the first case Lemma \ref{lemma:curvature} (a) gives the correct curvature estimate.\\
	If $\tilde q$ does not belong to one of the $k_{j+1}$ copies of $\hat M_j$, then $\tilde q \in S^3_{\delta_{j+1}r_{j+1}} \times C(S^2_{1-\epsilon_{j+1}})$ and 
	\begin{equation}\label{z12}
		|Rm|(q) \le \frac{10}{r_{j+1}^2\delta_{j+1}^2} \, , 
	\end{equation}
	provided $\epsilon_{j+1}\le 1/5$. On the other hand, the gluing region $A_{r_{j+1},10r_{j+1}}(p_j)\subset \hat M_j$ is isometric to an annulus of $S^3_{\delta_{j+1}r_{j+1}} \times C(S^2_{1-\epsilon_{j+1}})$. Hence 
	\begin{equation}
	|Rm|(x)\ge \frac{10^{-10}}{(r_{j+1} \delta_{j+1})^2}\, ,
	\end{equation}
	for any $x\in A_{r_{j+1},10r_{j+1}}(p_j)$.\\
	By Lemma \ref{lemma:curvature} this curvature must be bounded by $\frac{\mu}{r_{j+1}^{2-\eta}}$. Since $d(\pi_{j+1}(p_{j+1}),q) \le k_{j+1} r_{j+1}$, up to choosing $\mu\ll 1$, \eqref{z12} allows us to complete the proof of the inductive step and hence of Proposition~\ref{prop:curvdecay}. $\qed$

	\begin{proof}[Proof of Lemma \ref{lemma:curvature}]
	
	The key idea of the proof is to slightly modify the construction of the neck region in \cite[Section 7]{BrueNaberSemolaMilnor1} by increasing the number of scales where the space is isometric to a cone over $S^3 \times S^2$. In this region, the curvature decays quadratically and this will arbitrarily improve the sub-critical curvature estimate \eqref{eq:eta_curvature}. 
	
	Let us begin by recalling the notation from \cite[Section 7]{BrueNaberSemolaMilnor1}. Our (scaling invariant) neck region $X$ is obtained by gluing together seven different pieces. For the sake of this proof, it is enough to group them into two components:
	\begin{itemize}

		\item $X_1\cup X_2 \cup X_3$ is an annular region whose first end is isometric to an annulus in $S^3_{\delta_j} \times C(S^2_{1-\epsilon_j})$ (this is where the gluing with $M_j$ takes place), and the second end is isometric to an annulus in $C(S^3_{\delta(2^j)} \times S^2_{1/8})$. The parameter $\delta(k_{\le j})$ is chosen small enough to accommodate the next gluing, and it is determined by the application of \cite[Lemma 7.2]{BrueNaberSemolaMilnor1} with $k=k_{\le j}$.
		
		\item $X_4 \cup X_5 \cup X_6 \cup X_7$ is the annular region where the equivariant twisting takes place, its second end is isometric to an annulus in $S^3_{R_j\hat \delta_j} \times C(S^2_{1-\hat \epsilon_j})$.
	\end{itemize}
	
	We notice that the transition region $X_1\cup X_2$ requires a uniformly bounded number of scales, independent of $j$. On the other hand, the transition region $X_3$ requires a number of scales depending on $\delta_j/\delta(k_{\le j})$, which is not uniformly bounded a priori, and this might be problematic for controlling the curvature.\\

	To get the sought curvature bounds we need to modify the construction as follows. We build the first transition region $X_1 \cup X_2 \cup X_3$ from $S^3_{\delta_j} \times C(S^2_{1-\epsilon_j})$ to $C(S^3_{\delta} \times S^2_{1/8})$ by choosing $\delta= \delta_j/10$, which might be much larger than the parameter $\delta(2^j)$ in the original construction. This can be done in at most $10$ scales and worsening the curvature by at most a factor of $100$. 
	
	Observe that by the inductive assumption \eqref{claim2}, we have an improved estimate on the curvature of the first end of $X_1 \cup X_2 \cup X_3$. Now, on the second end of $X_1 \cup X_2 \cup X_3$ the space is isometric to an annulus in  $C(S^3_{\delta_j/10} \times S^2_{1/8})$ and still satisfies the correct non-scale-invariant curvature estimate as in \eqref{claim1}. In particular, we can glue in an annulus $A_{10, 10R}(O) \subset C(S^3_{\delta_j/10} \times S^2_{1/8})$. The resulting space $X_1 \cup X_2 \cup X_3 \cup Y$ satisfies the non-scale-invariant curvature estimate. More than this, if $R=R(\mu)$ is big enough, close to its second end the estimate improves arbitrarily to
	 \begin{equation}
	 	|Rm|(q) \le \frac{\mu}{d(p,q)^{2-\eta}} \, ,
	 \end{equation}
	for any $\mu>0$, as the curvature has the faster quadratic decay on a cone.\\
	 We pick $\mu$ small enough in order to be able to perform the remaining gluing while keeping the curvature estimate as in \eqref{eq:eta_curvature}. This can be done because the curvature on the next annular region is scaling invariantly bounded by a constant depending only on $k=k_{\le j}$. We remark that this region should include an additional transition from $C(S^3_{\delta_j/10} \times S^2_{1/8})$ to $C(S^3_{\delta(k_{\le j})} \times S^2_{1/8})$ with respect to the original construction. However, this transition can be handled with techniques analogous to those entering the other steps of the construction.
	\end{proof}

	\vspace{.3cm}

\subsection{Volume of balls}

Our goal in this section is to show that the counterexamples to the Milnor conjecture can be constructed as to have volume of the unit balls bounded away from zero and to discuss their volume growth.  That is, we will prove Theorem \ref{t:milnor6d}.2-\ref{t:milnor6d}.4 .

\subsubsection{Unit scale non-collapsing}

Let us begin by addressing Theorem \ref{t:milnor6d}.4:

\begin{proposition}\label{prop:1noncollapsing}
The complete manifold $(M^6,g)$ as in Theorem~\ref{t:milnor6d} can be constructed so that it satisfies
\begin{equation}
\inf_{q\in M}\mathrm{vol}(B_1(q))>0\, .
\end{equation}
\end{proposition}

\begin{proof}
We argue by induction as in the proof of Proposition \ref{prop:curvdecay} and borrow the notation introduced therein.

We notice that the parameters in the constructions can be chosen so that the displacement of any point with respect to any isometry $\gamma\in\Gamma$ is greater or equal to $2$.  In particular, it is enough to show the result on the universal cover.\\
Recall that $M_{j+1}$ is obtained by gluing $k_{j+1}$ copies of $\hat M_j$ into a slight perturbation of $S^3_{\delta_{j+j}r_{j+1}} \times C(S^2_{1-\epsilon_{j+1}})$ after removing $k_{j+1}$-copies of $S^3 \times D^3$, see Proposition~\ref{p:step3}
 for the precise construction. If $\tilde q$ does not belong to one of the copies of $\hat M_j$, then $\mathrm{vol}(B_1(q))>\frac{1}{100}$ provided $\delta_{j+1} r_{j+1}/k_{\le j} \ge 1$. We can always make the latter choice of parameters.\\
If $\tilde q$ belongs to one of the copies of $\hat M_j$, then we distinguish two cases. When $\tilde q \in M_j \subset \hat M_j$, the conclusion follows by inductive assumption. When $\tilde q$ belongs to the neck region $\hat M_j \setminus M_j$ we use that the metric is explicit everywhere, except in the region that was denoted by $X_4$ in \cite[Section 7]{BrueNaberSemolaMilnor1}. However, also in $X_4$ the volume of unit balls is uniformly bounded below provided $r_{j+1}\gg 1$ is big enough.

\end{proof}

\subsubsection{Volume of big balls}

The goal of the next proposition is to show that the counterexamples to the Milnor conjecture can be constructed so that their volume growth is almost maximal.  We will prove Theorem \ref{t:milnor6d}.2 and Theorem \ref{t:milnor6d}.3 :

\begin{proposition}\label{prop:volumegrowth}
For every $\eta>0$, the complete manifold $(M^6,g,p)$ as in Theorem~\ref{t:milnor6d} can be constructed so that it satisfies
\begin{equation}\label{eq:uppervolgrowth}
\mathrm{vol}(B_{s_i}(p))=s_i^{6-\eta}\, ,
\end{equation}
for some sequence $s_i\to \infty$, and 
\begin{equation}\label{eq:lowervolgrowth}
\mathrm{vol}(B_{t_i}(p))=t_i^{3+\eta}\, ,
\end{equation}
for some sequence $t_i\to \infty$.

An analogous statement holds for the universal cover $(\tilde{M},g,\tilde{p})$.
\end{proposition}

\begin{proof}
We provide an argument only for the volume growth of $(M,g)$, the argument for the universal cover being completely analogous.  We borrow again the notation from \cite[Section 7]{BrueNaberSemolaMilnor1}. \\

Notice that $(M,g)$ contains domains isometric to annuli $O_i,W_i$ with $O_i\subset C\left(\left({\dZ_{k_{\le i}}\backslash S^3_{\lambda_i}}\right)\times S^2_{\xi_i}\right)$ and $W_i\subset   \left({\dZ_{k_{\le i}}\backslash S^3_{\lambda_i}}\right)\times C(S^2_{1-\eta_i})$, where $\lambda_i,\eta_i,\xi_i,\eta_i>0$, for every $i\in\dN$. Here $\dZ_{k_{\le i}}\backslash S^3$ denotes the quotient of $S^3$ with respect to the action of $\dZ_{k_{\le i}}\subset S^1$, where $S^1$ acts by left Hopf rotation.  At the level of the universal cover $\tilde{M}$, these regions correspond to annuli between the regions $X_4$ and $X_5$ and an annulus at the end of the regions $X_7$, respectively. Notice in particular that the action of $\Gamma_i$ is by pure Hopf-rotation on the $S^3$ factor in those regions.
Moreover, in these annuli a straightforward computation shows the volume growth is $\sim r^6$ and $\sim r^3$ respectively, up to constant multiplicative coefficients. 

It is then elementary to show that there exist sequences $s_i,t_i\to \infty$ such that \eqref{eq:uppervolgrowth} and \eqref{eq:lowervolgrowth} hold, provided that the annular regions $O_i$ and $W_i$ above are chosen to be sufficiently large. 
This can be accomplished with a slight modification of the construction in \cite[Section 7]{BrueNaberSemolaMilnor1}, analogous to the one discussed in the proof of Lemma \ref{lemma:curvature} above. More precisely, we can insert an arbitrarily large region where $\tilde{M}$ is isometric to an annulus in $C(S^3\times S^2)$ between the regions $X_4$ and $X_5$ at every step of the inductive construction. Analogously, we can insert an arbitrarily large region where $\tilde{M}$ is isometric to an annulus in $S^3\times C(S^2)$ at the end of $X_7$ at each step of the inductive construction.
\end{proof}

\begin{remark}
The first part of the statement in Proposition~\ref{prop:volumegrowth} above should be compared with a result of B.-Y. Wu (see \cite{Wu}) saying that if $\alpha\leq \alpha(n)$ is such that $(M^n,g)$ has $\Ric\ge 0$ and the limit
\begin{equation}
\lim_{r\to \infty}\frac{\mathrm{vol}(B_r(p))}{r^{n-\alpha}}
\end{equation} 
exists and is strictly positive, then $\pi_1(M)$ is finitely generated. The effect of Proposition~\ref{prop:volumegrowth} is to show that the limit in the assumptions of \cite[Theorem 1.2]{Wu} cannot be replaced by a limsup.\\
\end{remark}

\subsubsection{\bf Tangent Cones at Infinity of $\tilde M$ and $M$}\label{ss:tangentcones}

Let us consider a sequence of radii $s_j\to \infty$ and understand the limits of $(s_j^{-1}\tilde M,p,\Gamma)$ and $(s_j^{-1}M,p)$.  After passing to subsequences (and reindexing) we can break ourselves down into various cases depending on how $s_j$ compares to our naturally defined scales $r_j$ from before. 

As we shall discuss below, the family of tangent cones that can appear will be analogous to those appearing for the $7$-dimensional counterexamples to the Milnor conjecture constructed in \cite{BrueNaberSemolaMilnor1}. The only difference, besides the obvious changes of dimensions, will be that the cross sections of some of the tangent cones at infinity will be suspensions over circles $S^2/\dZ_k$, rather than being lens spaces $S^3/\dZ_k$.

\subsubsection{The scales $s_j=r_j$}

Let us begin with the base case of understanding the sequence  $(r_j^{-1}\tilde M, p, \Gamma)$ on the universal cover.  We have determined that $\tilde M$ looks very close to $S^3\times \dR^3$ at these scales with (scale invariantly) shrinking sphere factor.  In particular, we have that geometrically the tangent cone at infinity along this sequence gives $r_{j}^{-1}\tilde M\to \dR^3$.  The action of $\gamma_j$ at scale $r_j$ is visible as a rotation by angle $2\pi/k_j$ of the $\dR^3$ factor with respect to a basepoint distance $k_j$ away. Therefore to understand the equivariant limit we need to break ourselves into two cases.  Namely, after passing to subsequences either $k_j$ converges or not.\\

\subsubsection{The scales $s_j=r_j$ with $k_j\to k<\infty$}\label{sss:kjtok}

In this case the action of $\gamma_j$ looks like a rotation with respect to a point distance $kr_j$ away from $p$, and so we have that $(r_j^{-1}\tilde M, p, \Gamma)\to (\dR^3,p_\infty,\dZ_k)$ where $\dZ_k$ is acting by rotation around the origin and $p_\infty$ is a point distance $k$ from the origin.  We get that the quotient space 
\begin{align}
	(r_j^{-1}M,p)\to (C(S^2_1/\dZ_k),p_\infty)
\end{align}
limits to a cone over the spherical suspension over a circle of length $2\pi/k$. This cone is isometric to $\dR\times C(S^1_{1/k})$. The basepoint $p_\infty$ of this limit is again a point distance $k$ from the cone point.

\subsubsection{The scales $s_j=r_j$ with $k_j\to\infty$}

In this case the action of $\gamma_j$ is looking increasingly like a translation by $\dZ$, and we get that $(r_j^{-1}\tilde M, p, \Gamma)\to (\dR^3,0,\dZ)$ where $\dZ$ acts by unit translation.  The quotient space in this case limits 
\begin{align}
	r_j^{-1}M\to \dR^2\times S^1\, .
\end{align}

\subsubsection{The scales $r_j<s_j<<k_jr_j$ with $k_j\to \infty$} In the case that $k_j\to k$ remains bounded there is no distinction between this case and the last.  Therefore, we are only concerned with the case where we have some subsequence for which $k_j\to \infty$.  In this situation note with $\frac{s_j}{r_j},\frac{k_jr_j}{s_j}\to \infty$ that our $\dZ$ action is looking increasingly like an $\dR$ action.  Our limit in this case becomes $(s_j^{-1}\tilde M, p, \Gamma)\to (\dR^3,0,\dR)$, where $\dR$ is acting by translation.  Our quotient space is therefore limiting
\begin{align}
	s_j^{-1}M\to \dR^2\, .
\end{align}

\subsubsection{The scales $s_j\approx k_jr_j$ when $k_j\to \infty$}  Note the action of $\gamma_j$ at these scales looks like a rotation by angle $2\pi/k_j$.  In particular, we get that $(s_j^{-1}\tilde M, p, \Gamma)\to (\dR^3,p_\infty,S^1)$, where $S^1$ is a rotation around the origin.  Our basepoint is now roughly distance $1$ from the center of the rotation.  In particular our quotient limit is given by
\begin{align}
	(r_j^{-1}M,p)\to (C([0,\pi]),p_\infty)\, ,
\end{align}
where we denoted by $C([0,\pi])$ the cone over the interval, which is isometric to the half-plane $\dR^2_{+}$.

\subsubsection{The scales $k_jr_j<<s_j << r_{j+1}$ when $k_j\to k<\infty$} We discussed that at scale $s_j\approx k_jr_j$ we have $s_j^{-1}\tilde M$ looks like $\dR^3=C(S^2_1)$.  As $\frac{s_j}{k_j r_j}$ increases our cross section sphere $S^2_s$ begins to decrease in radius until it looks like a half ray.  Therefore we get the possible limits $(s_j^{-1}\tilde M,p,\Gamma)\to (C(S^2_s),0,\dZ_k)$ for all $0\leq s\leq 1$.  In the case when $\frac{s_j}{k_jr_j}$ becomes sufficiently large we get that the limit is a half ray with the trivial action.  Our quotient limits in this range are therefore
\begin{align}
	(s_j^{-1}M,p)\to (C(S^2_s/\dZ_k),p_\infty)\, ,
\end{align}
for all $0\leq s\leq 1$.\\

\subsubsection{The scales $s_j\to r_{j+1}$}
As the scale $s_j$ continues to increase to $r_{j+1}$, we have that the half ray reopens up so that we again have $s_j^{-1}\tilde M\approx \dR^3$.  However, as it reopens the $\Gamma_j$ is now a trivial action.   As we approach scale $r_{j+1}$ a new $\gamma_{j+1}$ action appears and we repeat the above process.\\

In the case $\Gamma=\dQ/\dZ$ we can choose $k_j$ so that every $k\in \dN$ appears infinitely often.  Consequently, all of the cones
\begin{align}
	M_\infty \equiv C(S^2_s/\dZ_{k})\, ,
\end{align}
appear as tangent cones at infinity for all $s\in [0,1]$ and $k\in \dN$. \\

The last point to remark on is that there are some tangent cones at infinity which are metric cones, though the pointed limit does not have the cone point as the base point, as it happens for the examples constructed in \cite{BrueNaberSemolaMilnor1}. In this regard, recall that in \cite{Sormanilinear} it was proved that for a manifold with $\Ric\ge 0$ and infinitely generated fundamental group some tangent cones at infinity need to not be polar {\it with respect to} the base point. Note also that the tangent cones isometric to $\R^2\times\mathbb{S}^1$ are not polar as well.\\

\section{Preliminaries on Riemannian submersions}\label{s:preliminaries}

In this Section we record some background material about Riemannian submersions that will turn out to be important for the proof of Theorem \ref{t:equivariant_mapping_class_S3xS2}.

\subsection{Riemannian Submersions}\label{ss:prelim:submersions}

 Our setup is that we have Riemannian manifolds $(M^n,g)$ and $(B,g_b)$ together with a Riemannian submersion
\begin{align}
	\pi: M\stackrel{F}{\longrightarrow} B\, .
\end{align}

Throughout we will let $U,V,..$ denote vertical vector fields on $M$, so $U,V\in TF\equiv \mathcal V\subseteq TM$, and we will let $X,Y,..$ denote horizontal vector fields on $M$, so $X,Y\in T^\perp F\equiv \mathcal H\subseteq TM$. The integrability tensor of the Riemannian submersion is defined by
\begin{equation}
	A_{E_1}E_2:=\mathcal{H}\nabla_{\mathcal{H}E_1}\mathcal{V}E_2+\mathcal{V} \nabla_{\mathcal{H}E_1}\mathcal{H}E_2\, ,
\end{equation}
where our notation $\mathcal V E$ and $\mathcal H E$ denote the projections of $E$ to the corresponding subspaces, see \cite[Definition 9.20]{Besse}.  Recall that if $X,Y$ are horizontal vector fields then
\begin{equation}
	A_{X}Y=\frac{1}{2}\mathcal{V}[X,Y]\, .
\end{equation}

For the proposition below we refer the reader to O' Neill \cite{Oneill} (see also \cite[Proposition 9.36]{Besse}):

\begin{proposition}[Ricci curvature for Riemannian submersions]\label{prop:Rictotgeo}
	Let $\pi:(M,g)\to (B,g_B)$ be a Riemannian submersion with totally geodesic fibers $F$. Then 
	\begin{align}
		\Ric_M(U,V)=&\, \Ric_F(U,V)+(AU,AV)\, ,\\
		\Ric_M(U,X)=&\, \left({\rm div}_BA[X],U\right) \, ,\\
		\Ric_M(X,Y)=&\, \Ric_B(X,Y)-2(A_{X},A_{Y})\, ,
	\end{align}
	where $ \Ric_F$ stands for the Ricci curvature of the fiber with the induced Riemannian metric and $\Ric_B$ is the Ricci curvature of the base, understood as a horizontal tensor on $M$.  
\end{proposition}
\begin{remark}
	Note that in the above proposition we have the explicit expressions
	\begin{align}
		&(AU,AV):=\sum_ig(A_{X_i}U,A_{X_i}V)\, ,\notag\\
		&(A_X,A_Y):=\sum_ig(A_{X}X_i,A_{Y}X_i)\, ,\notag\\
		&\text{div}_BA:=\sum_i\left(\nabla_{X_i}A\right)(X_i,\cdot)\, ,
	\end{align}
	where $\{X_i\}$ is an orthonormal basis of the horizontal space.
\end{remark}

It is helpful to record how the Ricci curvature on the total space of the Riemannian submersion changes when we perform the so called \emph{canonical variation} of the metric, i.e. we define $g_t$ by leaving the horizontal distribution unchanged, the metric on the base unchanged, and scaling the metric on the fibers by a factor $t$. Below we shall assume again that the fibers are totally geodesic, see \cite[Proposition 9.70]{Besse}.

\begin{corollary}\label{cor:canonicalvariation}
	Let $\pi:(M,g)\to (B,g_B)$ be a Riemannian submersion with totally geodesic fibers and let $g_t$ the Riemannian metric on $M$ obtained by scaling the fibers metrics with a factor $t$. Then
	\begin{align}
		\Ric_t(U,V)=& \, \Ric_F(U,V)+t^2(AU,AV)\, ,\\
		\Ric_t(X,U)= & \, t \left({\rm div}_BA[X],U\right)\, , \\
		\Ric_t(X,Y)=& \, \Ric_B(X,Y)-2t\left(A_X,A_Y\right)\, .
	\end{align}
	Above, $A$ denotes the integrability tensor of the Riemannian submersion $\pi:(M,g)\to (B,g_B)$.
\end{corollary}

We are going to rely on the following technical lemma about Riemannian submersions with totally geodesic fibers over oriented surfaces.

\begin{lemma}\label{lemma: div}
	Let $\pi : (N,g_N) \to (\Sigma^2,g_{\Sigma})$ be a smooth Riemannian submersion with totally geodesic fibers, where $(\Sigma, g_\Sigma)$ is a compact, oriented surface. Let $\omega := \pi^* \Vol_{\Sigma}$. Then, ${\rm div}_N \omega$ vanishes along horizontal directions and satisfies 
	\begin{align}
		{\rm div}_N \omega[V] = 
		2 \langle  A(X_1,X_2), V \rangle 
	\end{align}
	for every vertical direction $V$, where $X_1$, $X_2$ are horizontal vector fields corresponding to a local oriented orthonormal frame of $\Sigma$.
\end{lemma}

\begin{proof}
	For every $Z$, we compute
	\begin{equation}\label{eq:divform}
		{\rm div}_N \omega[Z] = \sum_{i=1}^2 \nabla_{X_i} \omega [X_i,Z] + \sum_{j} \nabla_{U_j} \omega[U_j,Z]\, ,
	\end{equation}
	where $U_j$ is an orthonormal frame of the vertical space. 
	\medskip
	
	Using that $\pi_*[U_j] = 0$ and that the fibers are totally geodesic, we can deduce that $\sum_{j} \nabla_{U_j} \omega[U_j,Z] = 0$.
	
	\medskip

	When $Z=Y$ is horizontal, we have
	\begin{align}
	\nonumber	\nabla_{X_i} \omega [X_i,Y]
		& = X_i(\pi^*\Vol_{\Sigma}[X_i,Y])
		- \pi^* \Vol_{\Sigma}(\nabla_{X_i}X_i,Y)
		- \pi^* \Vol_{\Sigma}(X_i, \nabla_{X_i}Y)
		\\
		& = \nabla_{\pi_* X_i} \Vol_{\Sigma}[\pi_* X_i, \pi_* Y] \, .
	\end{align}
	Therefore,
	\begin{equation}
		{\rm div}_N \omega[Y] = {\rm div}_\Sigma \Vol_{\Sigma}[\pi_* Y] = 0 \, .
	\end{equation}

	\medskip

	When $Z=V$ is vertical, we get
	\begin{equation}\label{eq:divsimpl}
		\nabla_{X_i} \omega [X_i, V]
		=
		- \pi^* \Vol_{\Sigma}[X_i, \nabla_{X_i} V]
		= - \pi^* \Vol_{\Sigma}[X_i, A_{X_i} V] \, .
	\end{equation}
	From the identity
	\begin{equation}
		\langle  A_{X_i}V , X_j \rangle  = - \langle A[X_i,X_j], V \rangle  \, ,
	\end{equation}
see \cite[Equation (9.21c)]{Besse}, we deduce that
	\begin{align}\label{eq:usefform}
		&A_{X_1} V = - \langle A[X_1, X_2], V \rangle  X_2
		\\
		& A_{X_2} V = \langle A[X_1,X_2], V \rangle X_1 \, .
	\end{align}
	Hence, from \eqref{eq:divform}, \eqref{eq:divsimpl} and \eqref{eq:usefform} we conclude
	\begin{align}
	\nonumber	{\rm div}_N\omega(V)&=\sum_{i=1}^2 \nabla_{X_i} \omega [X_i, V]\\  
	\nonumber 	&= 2 \langle  A[X_1,X_2], V \rangle  \Vol_{\Sigma}(\pi_* X_1, \pi_* X_2)
		 \\
		& = 2 \langle A[X_1,X_2], V \rangle \, .
	\end{align}
\end{proof}

\subsection{Riemannian Submersions and Circle Bundles}

Let us now restrict ourselves to the case of a Riemannian $S^1$-principal bundle, so that $\pi:M\to B$ is the total space of an $S^1$-principal bundle over $B$.\\ 
Note that if $(B,g_B)$ is a Riemannian manifold, then an $S^1$-invariant metric on $M$ is well defined by the additional data of a principal connection $\eta\in \Omega^1(M)$ and a smooth $f:B\to \dR^+$ which prescribes the length of the $S^1$ fiber above a point.  If $\partial_t$ is the invariant vertical vector field coming from the $S^1$ action, then we have the expressions
\begin{align}
	&\mathcal H = \ker\eta\, ,\notag\\
	&\eta[\partial_t] = 1\, ,\notag\\
	&g(\partial_t,\partial_t) = f^2\, .
\end{align}
In the case of an $S^1$ bundle we have that $d\eta = \pi^*\omega$ where $\omega\in \Omega^2(B)$ is the curvature $2$-form, which relates to the integrability tensor $A$ on $M$ by 
\begin{align}
	A(X,Y) = -\frac{1}{2}\omega[X,Y]\,\partial_t\, .
\end{align}

The following proposition is borrowed from \cite[Lemma 1.3]{GilkeyParkTuschmann}, where it was used to show that any principal $S^1$ bundle $\pi:M\to B$ admits an $S^1$-invariant metric of positive Ricci curvature when the base $(B,g_B)$ has positive Ricci curvature and the total space has finite fundamental group. 

\begin{proposition}\label{prop:RicS1bundle}
	Let $M\stackrel{S^1}{\longrightarrow} B$ be a Riemannian $S^1$-principal bundle as above with $X$ a unit horizontal vector and $U=f^{-1}\partial_t$ a unit vertical vector. Then 
	\begin{align}\label{eq:RicS^1warped}
	\Ric(U,U)=&\,   -\frac{\Delta f}{f}+ \frac{f^2}{2}|\omega|^2\, ,\\
	\label{eq:RicS^1warpedmixed}	\Ric(U,X)=&\, \frac{1}{2}\left(-f\left({\rm div}_B\omega\right)(X)+3\omega[X,\nabla f]\,\right)   \, \\
\label{eq:RicS^1warpedbase}		\Ric(X,X)=&\, \Ric_B(X,X)-  \frac{f^2}{2}|\omega[X]\,|^2 -\frac{\nabla^2f(X,X)}{f}\, ,
	\end{align}
	where it is understood, when necessary, that we are identifying the horizontal vector field $X$ with an element of $TB$.
\end{proposition}

Below we record a well-known lemma about Gauge transformations for principal $S^1$-bundles that will be useful later.

Recall that a Gauge transformation of an $S^1$-principal bundle $\pi: M \to B$ is a diffeomorphism $\Phi: M \to M$ such that $\pi\circ \Phi(p) = \pi(p)$ for every $p\in M$ and  
\begin{equation}
	\Phi(\theta \cdot p) = \theta \cdot \Phi(p) \, , \quad p\in M\, , \, \, \theta \in S^1\, .
\end{equation}

\begin{lemma}\label{lemma:Gauge}
	Any Gauge transformation $\Phi: M \to M$ of a simply connected  $S^1$-principal bundle $\pi: M \to B$ is isotopic to the identity.
\end{lemma}
\begin{proof}
	
	It is a classical property that there exists a smooth function $\theta:M\to S^1$ such that $\Phi(p)=\theta(p)\cdot p$ for every $p\in M$.
	
	Since $M$ is simply connected, we can lift $\theta: M \to S^1$ to the universal cover $\rho:\dR \to S^1$, obtaining $\hat \theta : M \to \dR$. Set $\theta_t(p):= \rho(t \hat \theta(p))$ for $t\in [0,1]$, $p\in M$. The map
	\begin{equation}
		\Phi_t (p) := \theta_t(p) \cdot p \, ,\quad t\in [0,1]\, ,
	\end{equation}
    produces the sought isotopy between the Gauge transformation and the identity.
\end{proof}

\vspace{.5cm}

\section{Equivariant twisting}\label{sec:equivtwist}

The goal of this section is to prove Theorem \ref{t:equivariant_mapping_class_S3xS2}, which involves several new and subtle points in comparison to our previous work in \cite{BrueNaberSemolaMilnor1} for the $7$ dimensional example.  We will restate the Theorem momentarily for the ease of readability.\\

First recall the following.  Let $k\in \mathbb{Z}$, then we endow $S^3 \times S^2$ with the (left) $(1,k)$-action
\begin{equation}
	\theta\cdot _{(1,k)}(s_1,s_2):= (e^{i\theta} \cdot s_1, e^{ik\theta}\cdot  s_2)\, , \quad \theta \in S^1 \, ,
\end{equation}
where $e^{i\theta }\cdot s_1$ indicates the left Hopf rotation in $S^3$, and $e^{ik\theta} \cdot s_2 := (e^{ik\theta}z,t)$ is rotation of $S^2$, where we identify $s_2 = (z,t)\in S^2 \subset \mathbb{C} \times \dR$.\\

Our aim is to show that when $k$ is even, there exists a smooth family of positively Ricci curved Riemannian metrics $(S^3\times S^2, g_t)$, $t\in [0,1]$, constant in a neighbourhood of the endpoints, invariant with respect to the $(1,k)$-action, and such that $g_0 = g_{S^3\times S^2}$ and $g_1 = \phi^* g_0$ where $\phi : S^3 \times S^2 \to S^3 \times S^2$ is a diffeomorphism satisfying
\begin{equation}
	\phi(\theta\cdot_{(1,k)}(s_1,s_2)) = \theta \cdot _{(1,0)}\phi(s_1,s_2)\, , \quad \theta \in S^1\, ,\quad (s_1,s_2)\in S^3\times S^2\, .
\end{equation}

Precisely:\\

\begin{theorem}\label{thm:5d twisting}
	Let $g_0=g_{S^3\times S^2}$ be the standard product metric on $S^3\times S^2$, and let $k\in \dZ$ be even.  Then there exist a diffeomorphism $\phi:S^3\times S^2\to S^3\times S^2$ and a family of metrics $(S^3\times S^2,g_t)$ such that
	\begin{enumerate}
		\item $\Ric_t>0$ for all $t\in [0,1]$
		\item The $S^1$-action $\cdot_{(1,k)}$ on $S^3\times S^2$ is an isometric action for all $g_t$.
		\item $g_1 = \phi^*g_0$ with $\phi\big(\theta\cdot_{(1,k)}(s_1,s_2)\big)= \theta\cdot_{(1,0)}\phi(s_1,s_2)$ .
	\end{enumerate}
\end{theorem}

For the rest of this section, we will be concerned with the proof of Theorem \ref{thm:5d twisting}. We refer to subsection \ref{subsec:outtwisting} below for an outline of the main steps of the proof.
\medskip

As a preliminary step, in subsection \ref{subsec:quotient} we are going to understand the geometry and the topology of the quotient $N:=S^1\backslash S^3\times S^2$, with respect to the $(1,k)$-action.

It turns out that for $k$ even the space $N$ is a (topologically) trivial Riemannian $S^2$ bundle over a round $S^2$ with totally geodesic fibers. However, the induced metric on the fibers is not round.\\ 
Moreover, the projection to the quotient space $\pi:S^3\times S^2\to N$ is associated with a Riemannian $S^1$-bundle whose fibers have non-constant length and whose connection is not Yang-Mills. This has the effect of introducing potentially negative terms in the formulas for the Ricci curvature of $S^1$-bundles from Proposition \ref{prop:RicS1bundle}.  In comparison to the construction of \cite{BrueNaberSemolaMilnor1} this will be the main source of technical difficulty for the proof of Theorem \ref{thm:5d twisting}, and resolving this issue is a careful balancing act.\\

\subsection{The geometry of $N=S^1\backslash S^3\times S^2$}\label{subsec:quotient}

We endow $S^3 \times S^2$ with the standard metric $g_0:=g_{S^3_1 \times S^2_1}$. Let $Z_1, Z_2, Z_3$ an orthonormal base of right invariant vector fields on the $S^3$ factor. We assume that $Z_1$ induces the left-Hopf action. We will write $Z_1^*, Z_2^*, Z_3^*$ to denote the dual frame.

On $S^2$, we introduce standard spherical coordinates
\begin{equation}
	\begin{cases}
		x = \cos \theta \sin \psi
		\\
		y =  \sin\theta \sin \psi
		\\
		z = \cos \psi\, .
	\end{cases}
\end{equation}

Let us now define $(N,h)$ to be the isometric quotient of $(S^3\times S^2, g_0)$ by the $(1,k)$ action. We have that $\pi_{(1,k)} : S^3 \times S^2 \to N$ is a principal $S^1$ bundle with invariant vertical vector field
\begin{equation}
	\partial_t = Z_1 + k \frac{\partial}{\partial \theta} \, ,
\end{equation}
and connection form
\begin{equation}
	\eta_0 := \frac{1}{1 + k^2 \sin^2\psi}  \left( Z_1^* + k \sin^2\psi d\theta  \right) \, .
\end{equation}
Notice that 
\begin{equation}\label{eq:sizeS1}
	g_{0}(\partial_t, \partial_t) = 1 + k^2 \sin^2\psi \, 
\end{equation}
is not constant, hence the fibers are not totally geodesic. Moreover, it is not hard to check that the curvature form $\omega_0 = d \eta_0$ is not harmonic. Equivalently, the connection $\eta_0$ is not Yang-Mills.\\

It is easy to check that $N$ is an $S^2$-bundle over $S^2$ with projection map $\pi: N \to S^2$ induced by $S^3 \times S^2 \ni (s_1,s_2) \to \pi_{\rm Hopf}(s_1) \in S^2$.

The following statement about the structure of this $S^2$-bundle appears to be known, see for instance \cite[Section 3.3]{BettiolKrishnan}. However, we sketch its proof for the sake of completeness as we were not able to locate a proof in the literature.\\

\begin{lemma}\label{lemma:diffeostruct}
The following hold:
\begin{enumerate}
	\item When $k\in\mathbb{Z}$ is even,  $\pi : N \to S^2$ is a trivial $S^2$-bundle and hence $N$ is diffeomorphic to $S^2 \times S^2$.
	\item When $k\in\mathbb{Z}$ is odd, $\pi:N\to S^2$ is a non trivial $S^2$-bundle and hence $N$ is diffeomorphic to $\mathbb{CP}^2\#\overline{\mathbb{CP}^2}$.
\end{enumerate}
\end{lemma}

\begin{proof}

		We decompose the $S^1$-bundle $\pi_{\rm Hopf} : S^3 \to S^2$ as
		\begin{equation}
			S^3 = D^2 \times S^1 \, \bigcup_\psi \, D^2 \times S^1 \, ,
		\end{equation}
		where the clutching map $\psi : \partial D^2 \times S^1 \to \partial D^2 \times S^1$ is given by $\psi(e^{i\alpha}, e^{i\beta}) = (e^{i\alpha} , e^{i(\alpha + \beta)})$. 
		The Hopf action corresponds to the rotation of the $S^1$ factor in $D^2\times S^1$.

		The above induces a decomposition of $S^3\times S^2$ as
		\begin{equation}
			S^3 \times S^2 = D^2 \times S^1 \times S^2\,  \bigcup_\psi\,  D^2 \times S^1 \times S^2 \, ,
		\end{equation}
	    where, with a slight abuse of notation, $\psi$ denotes also the radial extension of the map $\psi$ introduced above.
	    
	    \medskip

	   With respect to this decomposition, the $(1,k)$-action can be written as
	    \begin{equation}
	    	\theta \cdot _{(1,k)} (a, e^{i\beta}, s) = (a, e^{i\theta + \beta}, e^{ik \theta } \cdot  s) \, ,\quad \theta \in S^1\, ,
	    \end{equation}
        where $(a,e^{i\beta}, s)\in D^2 \times S^1 \times S^2$, and we recall that $e^{ik \theta } \cdot  s = (e^{ik\theta}z,t)$ with the identification $s=(z,t)\in S^2\subset \mathbb{C}\times \dR$.
    \smallskip

	    It is easy to check that 
	    \begin{equation}
	    	\pi_{(1,k)} : D^2 \times S^1 \times S^2 \to D^2 \times S^2\, ,
	    	\quad \pi_{(1,k)}(a, e^{i\beta}, s) = (a, e^{-ik\beta} s) \, ,
	    \end{equation}
	    is the quotient projection induced by the $(1,k)$-action. The gluing map $\psi$ descends to the quotient inducing a gluing map
	    \begin{equation}
	    	\phi_k : \partial D^2 \times S^2 \to \partial D^2 \times S^2\, , \quad \phi_k(e^{i\alpha},s) = (e^{i\alpha}, e^{-ik\alpha}\cdot s) \, .
	    \end{equation}
	    Therefore, $N$ is isomorphic to the $S^2$-bundle over $S^2$ obtained by gluing two copies of $D^2 \times S^2$ with the clutching map $\phi_k$.
	   
	    It is classical, see for instance \cite{Steenrod}, that the isomorphism class of $N$ as an $S^2$-bundle over $S^2$ depends only on the homotopy class of the clutching map, viewed as a map from $S^1$ to $\mathrm{Diff}(S^2)$:
	    \begin{equation}\label{eq:loop O(3)}
	    	S^1 \ni \alpha \mapsto e^{-ik\alpha} \cdot \,   \in \mathrm{SO}(3) \subset {\rm Diff}(S^2)\, .
	    \end{equation}

        Recall that $\mathrm{SO}(3)$ is diffeomorphic to $\mathbb{RP}^3$. In particular, $\pi_1(\mathrm{SO}(3))=\mathbb{Z}_2$.\\
        
        {\bf Claim:} The map in \eqref{eq:loop O(3)} is homotopic to the constant map when $k$ is even, and to $\alpha \mapsto e^{i \alpha} \cdot $, which is homotopically non-trivial, when $k$ is odd. \\

	 In order to prove the claim,  we identify every point $p\in \dR^3$ with a purely imaginary quaternion in $S^3$. For every $g\in S^3$, the map $A_g$ defined via $A_g (p) := g \cdot p \cdot g^{-1}$ belongs to $ \mathrm{SO}(3)$. Moreover, the mapping $S^3 \ni g \mapsto A_g \in \mathrm{SO}(3)$ is a covering of degree two. 
	    
	    Notice that $A_g$ with $g= e^{i\theta/2}$ represents a rotation of angle $\theta$. Hence, up to correctly identifying the axis of rotation, $e^{i k \alpha} \cdot \, \in \mathrm{SO}(3)$ coincides with $A_{e^{ik \alpha/2}}$, $\alpha\in [0, 2\pi]$. 
	    
	    The curve $[0,2\pi] \ni \alpha \mapsto e^{ik\alpha/2}\in S^3$ is a closed loop when $k$ is even, hence it projects to a homotopically trivial loop in $\mathrm{SO}(3)$. On the other hand, when $k$ is odd, the curve $[0,2\pi] \ni \alpha \to e^{ik\alpha/2}\in S^3$ connects $1 \in S^3$ to $e^{i \pi} = -1 \in S^3$, hence it projects to a homotopically non-trivial loop in $\mathrm{SO}(3)$. This completes the proof of the claim.\\

	    It is easy to see that when \eqref{eq:loop O(3)} is homotopic to the constant map $N$ is isomorphic to the trivial bundle $S^2 \times S^2$. 
	    In the other case, we write
	    \begin{equation}
	    	D^2 \times S^2 = D^2 \times S^1 \times [-1, 1] / \sim\, ,
	    \end{equation}
	    where $\sim$ identifies $D^2 \times S^1 \times \{-1\}$ and $D^2 \times S^1 \times \{1\}$ with $D^2 \times \{-1\}$ and $D^2\times \{1\}$, respectively.
	    After gluing two copies of $ D^2 \times S^1 \times [-1, 1] / \sim$ with the clutching map $\phi_k$ we obtain
	    \begin{equation}
	    	S^3 \times [-1, 1] / \sim \, ,
	    \end{equation}
	    where $\sim$ collapses the Hopf fibers of $S^3 \times \{-1\}$ and $S^3\times \{-1\}$ to $S^2 \times \{-1\}$ and $S^2\times \{1\}$, respectively. It is well-known that the latter manifold is diffeomorphic to $\mathbb{CP}^2\#\overline{\mathbb{CP}^2}$.

\end{proof}

\subsection{Outline of the proof of Theorem \ref{thm:5d twisting}}\label{subsec:outtwisting}

The argument will be divided into four main steps corresponding to different subsections.\\

As noted in subsection \ref{subsec:quotient}, the fibers of the $S^1$ bundle $\pi_{(1,k)} : S^3 \times S^2 \to N$ are not totally geodesic, and the connection form is not harmonic. These issues are the primary new sources of difficulty in our construction with respect to the analogous one in \cite{BrueNaberSemolaMilnor1}. The presence of a nontrivial warping function and a nonzero Yang-Mills term make it challenging to control the Ricci curvature during the process of metric deformation.  In particular, for large $k$ these negative Ricci contributions are very large.

The first trick to handle these issues is to shrink the size of the $S^2$ factor. This deformation maintains the Ricci curvature positive, keeps the $S^1$ symmetry, and reduces the effect of the negative contributions to the Ricci curvature coming from the non-constant size of the $S^1$ fiber and from the divergence of the curvature form.
In subsection \ref{subsec:squishS2}, we will study the geometry of the total space $S^3\times S^2$, viewed as an $S^1$ bundle with respect to the $(1,k)$-action, and of its quotient $N$, when we shrink the $S^2$ factor.\\

Provided that the size of the $S^2$ factor is sufficiently small, in subsection \ref{subsec:fiblengthconn} we will be able to modify the length of fibres in the $S^1$-bundle until it becomes constant and to move the connection $1$-form to the Hopf connection $\eta_1:=Z_1^*$, while keeping the Ricci curvature positive.
This is the key step and most delicate part of our proof. It is crucial to execute both modifications (adjusting the size of the fibers and altering the connection form) simultaneously; otherwise, the positivity of the Ricci curvature will not be preserved.\\

The goal of the next two steps will be to deform the geometry of $N$ until it becomes isometric to the product of two round spheres.\\
In subsection \ref{subsec:roundS2} we will deform the induced metric of the $S^2$ fibers on $N$ until it becomes round.\\ 
The last step will be to flatten the connection. Once this has been achieved while keeping the Ricci curvature positive in subsection \ref{subsec:flattenconnection},  it will be easy to verify that the induced geometry on $S^3\times S^2$ is isometric to the standard one via a diffeomorphism conjugating the $(1,k)$-action to the $(1,0)$-action and the proof will be completed.

\subsection{Squishing the $S^2$-factor}\label{subsec:squishS2}
For every $\alpha\in (0,1]$, we consider the metric $g^\alpha:= g_{S^3_1 \times S^2_\alpha}$ on $S^3 \times S^2$.

As in subsection \ref{subsec:quotient} above, we define $(N,h^\alpha)$ to be the isometric quotient of $(S^3\times S^2, g^\alpha)$ by the $(1,k)$ action. We have that $\pi_{(1,k)} : S^3 \times S^2 \to N$ is a principal $S^1$ bundle with invariant vertical vector field
\begin{equation}
	\partial_t = Z_1 + k \frac{\partial}{\partial \theta} \, ,
\end{equation}
whose fibers have length
\begin{equation}\label{eq:falphak}
	f^\alpha(k)^2 := g^\alpha(\partial_t, \partial_t) = 1 + \alpha^2 k^2 \sin^2 \psi \, .
\end{equation}
The induced connection form is given by
\begin{equation}
	\eta_0^\alpha := \frac{1}{f^\alpha(k)^2}  \left( Z_1^* + \alpha^2 k \sin^2\psi d\theta  \right) \, .
\end{equation}
In our coordinates, the horizontal space is spanned by the orthonormal frame
\begin{equation}
	Z_2\, , \quad 
	Z_3\, , \quad
	F_1^\alpha := \frac{1}{\alpha} \frac{\partial}{\partial \psi} \, , \quad
	F_2^\alpha := \frac{1}{f^\alpha(k)} \left(\alpha k \sin \psi\, Z_1 - \frac{1}{\alpha \sin \psi} \frac{\partial}{\partial \theta} \right) \, .
\end{equation}
We notice that $\pi: (N,h_\alpha) \to (S^2, g_{1/2})$ is a Riemannian submersion for every $\alpha\in (0,1]$. Moreover, $F_1^\alpha, F_2^\alpha$ are tangent to the fibers of the induced $S^2$-bundle, and $Z_2, Z_3$ span the horizontal distribution.

\begin{lemma}\label{lemma:fiberstotgeoN}
The fibers of $\pi: (N,h_\alpha) \to (S^2, g_{1/2})$ are totally geodesic. 
\end{lemma}

\begin{proof}
It is sufficient to observe that
\begin{equation}\label{eq:totgeod1}
	2 \langle \nabla_{F_i} F_j, Z_k \rangle
	=
	\langle [F_i,F_j], Z_k \rangle 
	+
	\langle [F_i,Z_k], F_j \rangle
	+
	\langle [F_j, Z_k], F_i \rangle 
	= 0 \, ,
\end{equation}
for every $i,j=1,2$, and $k=2,3$.
\end{proof}

Our next goal is to compute some of the relevant quantities in order to understand the Ricci curvature of the quotient space $(N,h_{\alpha})$. The computations will be important in the subsequent subsections when we will start deforming the geometry.  

\subsubsection{The curvature form}
For the curvature form induced by the Riemannian submersion $\pi_{(1,k)}: (S^3\times S^2,g_{\alpha})\to (N,h_{\alpha})$ it can be easily checked that 
\begin{equation}\label{eq:omega0alpha}
	\omega_0^\alpha 
	:= d \eta_0^\alpha 
	= -\frac{2}{f^\alpha(k)^2} Z_2^* \wedge Z_3^* - \frac{2 k \cos \psi}{f^\alpha(k)^3} F_1^{\alpha,*} \wedge F^{\alpha,*}_2 \, ,
\end{equation}
where $\{F^{\alpha,*}_2\}$ is the dual basis of $\{F^{\alpha}_2\}$.
A direct computation shows that $\omega_0^\alpha$ is not divergence-free. More precisely
\begin{equation}\label{eq:divomega0alpha}
	- \frac{1}{2} {\rm div}_N \omega_0^{\alpha} = \frac{\alpha k \sin\psi}{f^\alpha(k)^3} \left(2 - \frac{1}{\alpha^2} - \frac{3k^2\cos^2\psi}{f^\alpha(k)^2}\right) F_2^{\alpha,*}\, .
\end{equation}

\subsubsection{The length of the fibres}
For the fibre length function $f^{\alpha}(k)$ (see \eqref{eq:falphak}) a direct computation shows

\begin{equation}\label{eq:gradfalphak}
	\nabla f^\alpha(k) = \frac{\alpha k^2 \sin \psi \cos \psi}{f^\alpha(k)} \, F_1^\alpha\, ,
\end{equation}

\begin{equation}\label{eq:laplfalphak}
	-\frac{\Delta f^\alpha(k)}{f^\alpha(k)}
	=
	\frac{k^2 \sin^2\psi}{f^\alpha(k)^2} - \frac{2 k^2 \cos^2\psi}{f^\alpha(k)^4}\, ,
\end{equation}

and 
\begin{equation}\label{eq:hessfalphak}
		f^\alpha(k)^{-1} \Hess f^\alpha(k)
		 =
		\left( \frac{k^2 \cos^2 \psi}{f^\alpha(k)^4} - \frac{k^2 \sin^2\psi}{f^\alpha(k)^2} \right)\, F_1^{\alpha, *}\otimes F_1^{\alpha, *}
		+
		\frac{k^2 \cos^2 \psi}{f^\alpha(k)^4}\, F_2^{\alpha, *}\otimes F_2^{\alpha, *}\, .
\end{equation}

\subsubsection{The Ricci tensor of $(N, h^\alpha)$}\label{subsec:Ric_N}

Here we record the expression for the Ricci tensor of $(N,h^{\alpha})$.

\begin{lemma}\label{lemma:RicNhalpha}
For any $\alpha>0$ the orthonormal basis $\{Z_2, Z_3, F_1^\alpha, F_2^\alpha\}$ diagonalizes the Ricci tensor of $N^{\alpha}:=(N, h^\alpha)$. Moreover
\begin{align}\label{eq:RicNZ}
	&\Ric_{N^\alpha}(F_1^\alpha, F_1^\alpha)
	 = \frac{1}{\alpha^2} + \frac{3k^2 \cos^2 \psi}{f^\alpha(k)^4} - \frac{k^2 \sin^2\psi}{f^\alpha(k)^2}\, ,
	 \\
\label{eq:RicNZF_2}	& \Ric_{N^{\alpha}}(F_2^\alpha, F_2^\alpha) 
	= \frac{1}{f^\alpha(k)^2}
	\left(2 \alpha^2 k^2 \sin^2\psi + \frac{1}{\alpha^2}\right) 
	+ \frac{3k^2 \cos^2\psi}{f^\alpha(k)^4}\, ,
	\\
\label{eq:RicNZZ_2}	&\Ric_{N^{\alpha}}(Z_2,Z_2) = \Ric_{N^{\alpha}}(Z_3,Z_3) = 2 + \frac{2}{f^\alpha(k)^2}\, .
\end{align}
\end{lemma}

\begin{proof}
By Proposition \ref{prop:RicS1bundle}, with a slight abuse of notation and the obvious identifications, we have
\begin{align}
\Ric_{N^{\alpha}}(X,Y)=&\Ric_{S^3\times S^2_{\alpha}}(X,Y)+\frac{f^\alpha(k)^2}{2}\omega_0^{\alpha}(X)\cdot\omega_0^{\alpha}(Y)\\
&+f^\alpha(k)^{-1} \Hess f^\alpha(k)(X,Y)\, .
\end{align}
The statement then follows from \eqref{eq:omega0alpha} and \eqref{eq:hessfalphak}.
\end{proof}

\subsection{Changing the length of the fibers and the connection}\label{subsec:fiblengthconn}

The goal of this subsection is to modify the geometry on $S^3\times S^2$ in order to make the length of fibres in the $S^1$-bundle $\pi:S^3\times S^2\to N$ constant and to reduce the connection to a ``standard'' one. These changes will leave the geometry on the base space $(N,h^{\alpha})$ unchanged.\\

Let us start by discussing this choice of preferred connection.

\begin{lemma}\label{lemma:defeta1}
The $1$-form
\begin{equation}
	\eta_1 := Z_1^* \, 
\end{equation}
is a principal connection for the principal $S^1$-bundle $\pi_{(1,k)}:S^3\times S^2\to N$
with associated curvature form
\begin{equation}
	\omega_1 := d\eta_1 = -2 Z_2^* \wedge Z_3^* \, .
\end{equation}
\end{lemma}

\begin{proof}
Notice that $\omega_0^\alpha - \omega_1 = d(- \frac{\alpha k \sin\psi}{f^\alpha(k)} F_2^{\alpha,*})$, and $- \frac{\alpha k \sin\psi}{f^\alpha(k)} F_2^{\alpha,*}$ is the pull-back of a smooth differential form in $N$. 

Hence, $\omega_0^\alpha$ and $\omega_1$ represent the same cohomology class in $N$.
Given that $\pi: N \to S^2$ is induced by the Hopf projection $\pi_{\rm Hopf} : S^3 \to S^2$, it turns out that
\begin{equation}\label{eq:omega_1 vol}
	\omega_1 = 2\, \pi^* {\rm Vol_{S^2_{1/2}}} \, ,
\end{equation}
where we denoted by ${\rm Vol}$ the volume form.
\end{proof}

\begin{definition}
We let $g_1^\alpha$ be the unique Riemannian metric on $S^3\times S^2$ such that $\pi_{(1,k)}: (S^3\times S^2, g_1^\alpha) \to (N,h^\alpha)$ is a Riemannian submersion with totally geodesic fibers, $g_{1}^\alpha(\partial_t,\partial_t)=1$, and connection form $\eta_1$. 
\end{definition}

The aim of this section is to connect $(S^3 \times S^2, g_0^\alpha)$ to $(S^3 \times S^2, g_1^\alpha)$ with a smooth family of $S^1$-invariant Riemannian metrics $g_\beta^\alpha$, $\beta\in [0,1]$, with uniformly positive Ricci curvature. We will be able to achieve this provided that $\alpha<\alpha(k)$ is sufficiently small.

\begin{proposition}\label{prop:g0alphatog1alpha}
If $0<\alpha<\alpha(k)$ then there exists a smooth family of Riemannian metrics $(S^3\times S^2,g_{\beta}^{\alpha})$, $\beta\in[0,1]$, such that the following hold:
\begin{itemize}
\item[i)] $g_{\beta}^{\alpha}$ is invariant with respect to the $(1,k)$-action for any $\beta\in[0,1]$;
\item[ii)] $\pi_{(1,k)}:(S^3\times S^2,g_{\beta}^{\alpha})\to (N,h^{\alpha})$ is a Riemannian submersion for any $\beta\in[0,1]$;
\item[iii)] $\Ric_{g_{\beta}^{\alpha}}>0$ for any $\beta\in[0,1]$.
\end{itemize}
\end{proposition}

For every $\beta \in [0,1]$, we let $g_{\beta}^{\alpha}$ the unique Riemannian metric on $S^3\times S^2$ such that $\pi_{(1,k)} : (S^3 \times S^2, g_\beta^\alpha) \to (N, h^\alpha)$ is a Riemannian submersion with length of fibres and connection respectively given by
\begin{align}\label{eq:warpconnalphabeta}
	& f^\alpha((1-\beta) k)^2 = 1 + \alpha^2 (1-\beta)^2 k^2 \sin^2 \psi \, ,
	\\
	& \eta^\alpha_\beta := (1-\beta) \eta_0^\alpha + \beta \eta_1 \, .
\end{align}
We will denote $U_\beta := f^\alpha ((1-\beta)k)^{-1} \partial_t$ the unitary vertical directions for these metrics.\\

Our next goal is to prove that $\Ric_{g_\alpha^\beta} \ge 1/2$ for any $\beta\in [0,1]$, provided that $0<\alpha\le \alpha(k)$. This will prove Proposition \ref{prop:g0alphatog1alpha}, as (i) and (ii) follow from the very construction.

\subsubsection{Preliminary computations}

To begin with, we can write down an orthonormal frame for the horizontal space induced by the principal connection $\eta_\beta^\alpha$:
\begin{equation}\label{eq:ortohorgalphabeta}
	Z_2\, , \quad 
	Z_3\, , \quad
	F_1^{\alpha, \beta} := F_1^\alpha = \frac{1}{\alpha} \frac{\partial}{\partial \psi} \, , \quad
	F_2^{\alpha, \beta} := F_2^\alpha - \beta \frac{\alpha k \sin \psi }{f^\alpha(k)}\partial_t
	\, .
\end{equation}

An easy computation (cf. with \eqref{eq:omega0alpha} above) shows that the curvature form $\omega^{\alpha}_{\beta}$ satisfies
\begin{equation}\label{eq:curvalphabeta}
	\omega^\alpha_\beta = d \eta_\beta^\alpha
	=
	- 2 \left(\beta + \frac{1-\beta}{f^\alpha(k)^2}\right) Z_2^* \wedge Z_3^* - \frac{2(1-\beta) k \cos \psi}{f^\alpha(k)^3} F_1^{\alpha,\beta,*} \wedge F_2^{\alpha,\beta,*} \, ,
\end{equation}
hence,
\begin{align}\label{eq:normomegaalphabeta}
	& |\omega_\beta^\alpha[F_1^{\alpha,\beta}]|^2 
	= |\omega_\beta^\alpha[F_2^{\alpha,\beta}]|^2
	=
	(1-\beta)^2 \, \frac{4k^2 \cos^2 \psi}{f^\alpha(k)^6}\, ,
	\\
\label{eq:normomegaalphabetaZ_2}	&|\omega_\beta^\alpha[Z_2]|^2
	= 
	|\omega_\beta^\alpha[Z_3]|^2
	=
	4 \left(\beta + \frac{1-\beta}{f^\alpha(k)^2}\right)^2\, ,
	\\
\label{eq:normomegaalphabetatot}	&
	|\omega_{\beta}^\alpha|^2 
	= 4\left(\beta + \frac{1-\beta}{f^\alpha(k)^2} \right)^2 + 4 (1-\beta)^2 \, \frac{k^2 \cos^2 \psi}{f^\alpha(k)^6}\, .
\end{align}
Moreover (cf. with \eqref{eq:divomega0alpha} above)
\begin{equation}\label{eq:divomegaalphabeta}	
	{\rm div}_N \omega_{\beta}^\alpha 
	= (1-\beta) {\rm div}_N \omega_0^\alpha - \beta \frac{4 \alpha k \sin \psi}{f^\alpha(k)}\, F_2^{\alpha,\beta, *}\, .
\end{equation}

Analogously, we can compute the gradient, Laplacian, and Hessian of the warping function introduced in \eqref{eq:warpconnalphabeta}. From \eqref{eq:gradfalphak}, \eqref{eq:laplfalphak} and \eqref{eq:hessfalphak}, we obtain:
\begin{equation}\label{eq:gradalphabeta}
	\nabla f^\alpha((1-\beta)k) = \frac{\alpha (1-\beta)^2k^2 \sin \psi \cos \psi}{f^\alpha((1-\beta)k)} \, F_1^{\alpha,\beta}\, ,
\end{equation}

\begin{equation}\label{eq:Deltalphabeta}
	-\frac{\Delta f^\alpha((1-\beta)k)}{f^\alpha((1-\beta)k)}
	=
	\frac{(1-\beta)^2 k^2 \sin^2\psi}{f^\alpha((1-\beta)k)^2} - \frac{2 (1-\beta)^2k^2 \cos^2\psi}{f^\alpha((1-\beta)k)^4}\, ,
\end{equation}
and 
\begin{equation}\label{eq:Hessalphabeta}
	\begin{split}
		f^\alpha((1-\beta)k)^{-1}& \Hess f^\alpha((1-\beta)k)
		\\& =
		\left( \frac{(1-\beta)^2 k^2 \cos^2 \psi}{f^\alpha((1-\beta)k)^4} - \frac{(1-\beta)^2 k^2 \sin^2\psi}{f^\alpha((1-\beta)k)^2} \right)\, F_1^{\alpha, \beta, *}\otimes F_1^{\alpha, \beta, *}
		\\ & \quad\quad 
		+
		\frac{(1-\beta)^2 k^2 \cos^2 \psi}{f^\alpha((1-\beta)k)^4}\, F_2^{\alpha,\beta, *}\otimes F_2^{\alpha,\beta, *}\, .
	\end{split}
\end{equation}

\subsubsection{The Ricci curvature}

In order to complete the proof of Proposition \ref{prop:g0alphatog1alpha}, we compute the Ricci tensor of $(S^3 \times S^2, g_\beta^\alpha)$ and show that its eigenvalues are bounded from below by $1/2$ for any $\beta\in[0,1]$, provided that $\alpha<\alpha(k)$.
\medskip

We will use the formulas for the Ricci curvature of circle bundles from Proposition \ref{prop:RicS1bundle} in combination with the expression for the Ricci curvature of the base $\Ric_{N^{\alpha}}$ obtained in Lemma \ref{lemma:RicNhalpha}, the expression for the curvature form $\omega_\beta^\alpha$ (see \eqref{eq:curvalphabeta}), and the expressions of the gradient, the Hessian and the Laplacian of the warping function $f^\alpha((1-\beta)k)$ (see \eqref{eq:gradalphabeta}, \eqref{eq:Hessalphabeta}
 and \eqref{eq:Deltalphabeta}). \\
 
The first observation is that the only nonvanishing off-diagonal values of the Ricci tensor $\Ric_{g_{\beta}^{\alpha}}$ of $(S^3 \times S^2, g_\beta^\alpha)$ in the orthonormal frame $\{U_\beta, Z_2,Z_3, F_1^{\alpha,\beta}, F_2^{\alpha,\beta}\}$ are in the span of $\{U_{\beta},F_2^{\alpha,\beta}\}$.\\

From \eqref{eq:RicS^1warpedbase}, \eqref{eq:RicNZZ_2} and \eqref{eq:normomegaalphabetaZ_2} (notice from \eqref{eq:Hessalphabeta} that the Hessian vanishes in the $Z_i$ directions), we can compute
\begin{align}
	\Ric_{g_\beta^\alpha}(Z_2,Z_2) 
	&= \Ric_{N^{\alpha}}(Z_2 , Z_2) - \frac{f^\alpha((1-\beta) k)^2}{2}|\omega_\beta^\alpha[Z_2]|^2 
	\\&=
	2 + \frac{2}{f^\alpha(k)^2} - 2f^\alpha((1-\beta) k)^2 \left(\frac{1- \beta}{f^\alpha(k)^2} + \beta \right)^2\, .
\end{align}
Moreover, $\Ric_{g_\beta^\alpha}(Z_3,Z_3) = \Ric_{g_\beta^\alpha}(Z_2,Z_2)$.

It is straightforward to check then that
\begin{equation}
\Ric_{g_\beta^\alpha}(Z_2,Z_2)=\Ric_{g_\beta^\alpha}(Z_3,Z_3)\ge 1/2\, ,
\end{equation}
for any $\beta\in [0,1]$, provided that $\alpha\le\alpha(k)$.\\

Again from \eqref{eq:RicS^1warpedbase}, \eqref{eq:RicNZ}, \eqref{eq:normomegaalphabeta} and \eqref{eq:Hessalphabeta} we can compute
\begin{align}\label{eq:RicalphabetaF1}
	\Ric_{g_\beta^\alpha}(F_1^{\alpha,\beta}, F_1^{\alpha,\beta})
	& =
	\Ric_{N^{\alpha}} (F_1^{\alpha}, F_1^{\alpha})
	- \frac{f^\alpha((1-\beta) k)^2}{2}|\omega_\beta^\alpha[F_1^{\alpha,\beta}]|^2 
	\\& \quad
	- \frac{\Hess f^\alpha((1-\beta) k)}{f^\alpha((1-\beta) k)}(F_1^{\alpha,\beta}, F_1^{\alpha,\beta})\, .
\end{align}
Then we notice from \eqref{eq:normomegaalphabeta} and \eqref{eq:Hessalphabeta} that 
\begin{equation*}
\sup_{\alpha,\beta\in [0,1]}\,  \,\Big\{ \frac{f^\alpha((1-\beta) k)^2}{2}|\omega_\beta^\alpha[F_1^{\alpha,\beta}]|^2+\Big\rvert\frac{\Hess f^\alpha((1-\beta) k)}{f^\alpha((1-\beta) k)}(F_1^{\alpha,\beta}, F_1^{\alpha,\beta})\Big\rvert\Big\} <C_1(k)<\infty\, .
\end{equation*}
Hence from \eqref{eq:RicalphabetaF1} and \eqref{eq:RicNZ} we can estimate
\begin{equation}
\Ric_{g_\beta^\alpha}(F_1^{\alpha,\beta}, F_1^{\alpha,\beta})\ge \frac{1}{\alpha^2} -k^2\sin^2\psi- C_1(k) \ge 1/2\, ,
\end{equation}
for any $\beta\in[0,1]$, provided that $\alpha<\alpha(k)$.

\bigskip

We finally compute and estimate the Ricci tensor restricted to the span of $\{U_\beta, F_2^{\alpha,\beta}\}$.
By \eqref{eq:RicS^1warped}, \eqref{eq:Deltalphabeta} and \eqref{eq:normomegaalphabetatot}, for the $S^1$-fiber direction we have
\begin{align}
	\Ric_{g_\beta^\alpha}(U_\beta, U_\beta) & = - \frac{\Delta f^\alpha((1-\beta) k)}{f^\alpha((1-\beta) k)} + \frac{f^\alpha((1-\beta) k)^2}{2}|\omega_\beta^\alpha|^2
	\\&
	= 
	\frac{(1-\beta)^2 k^2 \sin^2\psi}{f^\alpha((1-\beta) k)^2} 
	- \frac{2 (1-\beta)^2 k^2 \cos^2\psi}{f^\alpha((1-\beta) k)^4}
	\\& \quad \quad + 2 f^\alpha((1-\beta)k)^2\left(\left(\beta+  \frac{1-\beta}{f^\alpha(k)^2} \right)^2 +  (1-\beta)^2 \, \frac{k^2 \cos^2 \psi}{f^\alpha(k)^6}\right)\, .
\end{align}
It is then elementary to estimate
\begin{equation}
\Ric_{g_\beta^\alpha}(U_\beta, U_\beta)\ge 
	1 + (1-\beta)^2 k^2 \sin^2 \psi \, ,
\end{equation}
provided that $\alpha\le \alpha(k)$.\\

For the remaining on-diagonal term, again by \eqref{eq:RicS^1warpedbase}, we compute

\begin{align}
	\Ric_{g_\beta^\alpha}(F_2^{\alpha,\beta}, F_2^{\alpha,\beta}) 
	& = \Ric_{N^{\alpha}}(F_2^{\alpha}, F_2^{\alpha}) - \frac{f^\alpha((1 - \beta) k)^2}{2}|\omega_\beta^\alpha[F_2^{\alpha,\beta}]|^2 
	\\&
	\quad \quad - \frac{\Hess f^\alpha((1-\beta) k)}{f^\alpha((1-\beta) k)}(F_2^{\alpha,\beta}, F_2^{\alpha,\beta})\, .
\end{align}
By \eqref{eq:RicNZF_2}, \eqref{eq:normomegaalphabeta} and \eqref{eq:Hessalphabeta} we can estimate
\begin{equation}
\Ric_{g_\beta^\alpha}(F_2^{\alpha,\beta}, F_2^{\alpha,\beta}) \ge \frac{1}{\alpha^2 f^\alpha(k)^2} - 10 k^2 \, ,
\end{equation}
for any $\alpha,\beta\in[0,1]$.\\

Finally, we compute and estimate the off-diagonal term with the help of Proposition \ref{prop:RicS1bundle}, \eqref{eq:curvalphabeta}, \eqref{eq:divomegaalphabeta} and \eqref{eq:gradalphabeta}:
\begin{align}
	\Ric_{g_\beta^\alpha}(U_\beta , F_2^{\alpha,\beta}) 
	& =  - \frac{f^\alpha((1-\beta) k)}{2} {\rm div}_{N^{\alpha}} \omega_\beta^\alpha[F_2^{\alpha,\beta}] + \frac{3}{2} \omega_\beta^\alpha[F_{2}^{\alpha,\beta},  \nabla f^\alpha((1-\beta) k)]
	\\&
	= -\frac{(1-\beta)}{\alpha^2} \frac{ f^\alpha((1-\beta) k)\alpha k \sin \psi}{f^{\alpha}(k)^3} + C_2(k)\, ,
\end{align}
where $|C_2(k)| \le 10 k^3$.
\medskip

It is then elementary to check that the determinant and the trace of the tensor $\alpha^2\Ric_{g^{\alpha}_{\beta}}$ restricted to the span of $\{U_{\beta},F_2^{\alpha,\beta}\}$ are both bigger than $1/2$ for any $\beta\in[0,1]$, provided that $\alpha\le \alpha(k)$. In particular, the eigenvalues of $\Ric_{g^{\alpha}_{\beta}}$ are bigger than $1$ provided that $\alpha\le \alpha(k)\ll 1$.

\medskip

This completes the proof of Proposition \ref{prop:g0alphatog1alpha}. $\qed$

\subsection{Rounding the fibers of $\pi: N \to S^2$}\label{subsec:roundS2}

At this stage, we have a Riemannian submersion $\pi_{(1,k)} : (S^3 \times S^2, g_1^\alpha) \to (N,h^{\alpha})$ induced by a principal $S^1$-bundle metric with totally geodesic fibers of length one and connection form $\eta_1$.\\

An orthonormal basis with respect to the metric $g_1^{\alpha}$ on $S^3\times S^2$ is given by:
\begin{equation}\label{eq:ortogalpha1}
	U:= Z_1 + k \frac{\partial}{\partial \theta}\, , \quad
	Z_2\, , \quad
	Z_3 \, , \quad
	F_1^\alpha:= \frac{1}{\alpha} \frac{\partial}{\partial \psi} \, , \quad
	F_2^\alpha:= - \frac{f^\alpha(k)}{ \alpha \sin \psi} \frac{\partial}{\partial \theta}\, .
\end{equation}

We understand from \eqref{eq:ortogalpha1} that $(N,h^{\alpha})$ has the structure of a Riemannian $S^2$-bundle over a round $S^2$ with totally geodesic fibers. However, the induced metric on the fibers is not round (see Lemma \ref{lemma:curvS2fibers} below for the expression of the Gaussian curvature for the induced metric on the fibers).

\medskip
Our goal in this subsection is to change the geometry $(N, h^\alpha)$ in order to make the fibers of the Riemannian submersion $\pi: N \to S^2$ isometric to round spheres. Meanwhile, we will maintain the connection and the length of fibres for the principal $S^1$-bundle $\pi_{(1,k)}:S^3\times S^2\to N$ fixed.\\

With this aim we introduce the family of metrics $g_\beta^\alpha$, with $\beta\in [1,2]$, defined by the orthonormal frame
\begin{equation}
	U=Z_1+k\frac{\partial}{\partial\theta}\, ,\quad Z_2\, , \quad
	Z_3 \, , \quad
	F_1^{\alpha,\beta} := \frac{1}{\alpha} \frac{\partial}{\partial \psi} \, , \quad
	F_2^{\alpha,\beta}:= -\frac{f^\alpha( k)}{f^\alpha((\beta-1)k)}\frac{1}{\alpha \sin \psi} \frac{\partial}{\partial \theta} \, .
\end{equation}
Moreover, we denote by $h^{\alpha}_{\beta}$, $\beta\in[1,2]$, the unique Riemannian metric on $N$ such that $\pi_{(1,k)}:(S^3\times S^2,g^{\alpha}_{\beta})\to (N,h^{\alpha}_{\beta})$ is a Riemannian submersion.

\begin{proposition}\label{prop:roundS2}
With the notation introduced above, the following properties hold provided that $\alpha<\alpha(k)$:
\begin{itemize}
\item[i)] $\Ric_{g^{\alpha}_{\beta}},\Ric_{h^{\alpha}_{\beta}}>0$ for any $\beta\in[1,2]$;
\item[ii)] $\pi:(N,h^{\alpha}_{\beta})\to (S^2, g_{S^2_{1/2}})$ is a Riemannian submersion with totally geodesic fibers for every $\beta\in[1,2]$. For $\beta=2$ the induced Riemannian metric on the $S^2$ fibers is round; 
\item[iii)] $\pi_{(1,k)}:(S^3\times S^2,g_{\beta}^{\alpha})\to (N,h^{\alpha}_{\beta})$ is a Riemannian principal $S^1$-bundle with totally geodesic fibers and principal connection $\eta_1$, for any $\beta\in[1,2]$.
\end{itemize}
\end{proposition}

The remainder of this section is aimed at proving Proposition \ref{prop:roundS2}.\\ 
Items (ii) and (iii) will follow from the very construction and we will focus on proving (i). In order to control the Ricci curvature of the base $(N,h^{\alpha}_{\beta})$ and of the total space $(S^3\times S^2,g^{\alpha}_{\beta})$ we will rely again on the formulas for the Ricci curvature of Riemannian submersions from Proposition \ref{prop:Rictotgeo} (see also Proposition \ref{prop:RicS1bundle}).

\subsubsection{The geometry of $(N,h_\beta^\alpha)$}

The base space $(N, h_\beta^\alpha)$ has the structure of an $S^2$-bundle over $S^2$, and the projection $\pi:(N,h^{\alpha}_{\beta})\to S^2_{1/2}$ is a Riemannian submersion.

In $(S^3\times S^2,g_{\beta}^{\alpha})$, $\{Z_2,Z_3,F_1^{\alpha,\beta},F_2^{\alpha,\beta}\}$ span the horizontal distribution. Moreover, $Z_2$ and $Z_3$ represent the (orthonormal) horizontal directions associated with the base $S^2$. On the other hand $F_1^{\alpha,\beta}$ and $F_2^{\alpha,\beta}$ induce a vertical orthonormal frame on $(N,h^{\alpha}_{\beta})$ with respect to the Riemannian submersion $\pi:(N,h^{\alpha}_{\beta})\to S^2_{1/2}$. \\

With the Koszul formula it is easy to check that the fibers of the Riemannian submersion $\pi:(N,h^{\alpha}_{\beta})\to S^2_{1/2}$ are totally geodesic for every $\alpha$ and $\beta$, see \eqref{eq:totgeod1} above for an analogous computation.\\

For the induced Riemannian metrics on the $S^2$ fibers we have the following:

\begin{lemma}\label{lemma:curvS2fibers}
The Ricci curvature of the $S^2$-fibers of the Riemannian submersion $\pi:(N,h^{\alpha}_{\beta})\to S^2_{1/2}$ is given by 
\begin{equation}
	\Ric_{\mathrm{fib}}=-\frac{1}{\alpha^2}\left(\frac{f^\alpha((\beta-1)k)\sin\psi}{f^{\alpha}( k)}\right)''\cdot \frac{f^{\alpha}( k)}{f^\alpha((\beta-1)k) \sin\psi}\, g_{\mathrm{fib}}\, .
\end{equation}
In particular, the induced metric is round if $\beta=2$ and it satisfies $\Ric_{\mathrm{fib}}\ge 1/(2\alpha^2)$ for any $\beta\in[1,2]$ provided that $\alpha\le\alpha(k)$.
\end{lemma}

In order to compute and estimate the Ricci curvature of $(N,h^{\alpha}_{\beta})$ with the formulas for Riemannian submersions from Proposition \ref{prop:Rictotgeo}, we compute the curvature of the induced connection:
\begin{equation}\label{eq:curvatureS2S2}
	A(Z_2,Z_3) 
	=  \frac{1}{2} \mathcal{V}[Z_2,Z_3] 
	=  \frac{f^\alpha((\beta-1)k)}{f^\alpha( k)} \alpha k \sin \psi\,  F_2^{\alpha,\beta} \, .
\end{equation}

Hence,
\begin{align}
	& (A_{Z_2},A_{Z_2}) = (A_{Z_3}, A_{Z_3}) = \frac{f^\alpha((1-\beta)k)^2}{f^\alpha( k)^2} \alpha^2 k^2 \sin^2 \psi\, ,
	\\
	&(A_{Z_2},A_{Z_3}) = 0\, ,
	\\
	&(AF_2, AF_2) = 2 \frac{f^\alpha((\beta - 1)k)^2}{f^\alpha(k)^2} \alpha^2 k^2 \sin^2 \psi\, ,
	\\
	&(AF_1, AF_1) = (AF_1, AF_2) = 0 \, .
\end{align}

\bigskip

Moreover, an easy application of the Koszul formula shows that $\mathcal H \nabla_{Z_i} Z_j = 0$ and $\nabla_{Z_i} F_2^{\alpha,\beta}$ is horizontal for $i,j=2,3$. Hence
\begin{align}
 {\rm div}_{S^2_{1/2}} A[Z_i]  & = \nabla_{Z_2}A[Z_2,Z_i] + \nabla_{Z_3}A[Z_3,Z_i]
 \\  \nonumber& = \nabla_{Z_2}(A[Z_2,Z_i]) + \nabla_{Z_3}(A[Z_3,Z_i]) 
  \\  \nonumber& \quad
 -A[\nabla_{Z_2} Z_2, Z_i] - A[\nabla_{Z_3}Z_3, Z_i]-
 A[Z_2, \nabla_{Z_2}Z_i]-
 A[Z_3, \nabla_{Z_3}Z_i]
 \\& = \nabla_{Z_2}(A[Z_2,Z_i]) + \nabla_{Z_3}(A[Z_3,Z_i]) \, ,
\end{align}
is horizontal for $i=2,3$.
Therefore, by Proposition \ref{prop:Rictotgeo}, the Ricci tensor $\Ric_{h_\beta^\alpha}$ is diagonal in any orthonormal frame of $(N,h^{\alpha}_{\beta})$ induced by $\{Z_2,Z_3,F^{\alpha,\beta}_1,F^{\alpha,\beta}_2\}$, with
\begin{align}
	&\Ric_{h_\beta^\alpha}(F_1^{\alpha,\beta},F_1^{\alpha, \beta})
	=
	\Ric_{\mathrm{fib}}(F_1^{\alpha,\beta},F_1^{\alpha, \beta})\, ,
	\\
	&\Ric_{h_\beta^\alpha}(F_2^{\alpha,\beta},F_2^{\alpha, \beta})
	= \Ric_{\mathrm{fib}}(F_2^{\alpha,\beta},F_2^{\alpha, \beta}) + 2\frac{f^\alpha((1-\beta)k)^2}{f^\alpha( k)^2} \alpha^2 k^2 \sin^2 \psi\, ,
	\\
	&\Ric_{h^\alpha_\beta}(Z_i,Z_i)
	= 4-2\frac{f^\alpha((1-\beta)k)^2}{f^\alpha(k)^2} \alpha^2 k^2 \sin^2 \psi\, .
\end{align}
Taking into account Lemma \ref{lemma:curvS2fibers}, it is elementary to check that $\Ric_{h^{\alpha}_{\beta}}\ge 3$ provided that $\alpha \le \alpha(k)$. 

\subsubsection{The Ricci curvature of $(S^3\times S^2, g_\beta^\alpha)$}

In order to complete the proof of Proposition \ref{prop:roundS2} it remains to compute and estimate the Ricci curvature of $(S^3\times S^2,g^{\alpha}_{\beta})$.\\

We consider the Riemannian submersion $\pi_{(1,k)}:(S^3\times S^2,g^{\alpha}_{\beta})\to (N,h^{\alpha}_{\beta})$. Its fibers are totally geodesic and the induced curvature $2$-form is $\omega_1=-2Z_2^*\wedge Z_3^*$.\\

Since $\Ric_{h_\beta^\alpha} \ge 3$, as we established above, from \eqref{eq:RicS^1warped} we deduce that
\begin{equation}
	 \Ric_{g_\beta^\alpha}(U,U) = \frac{1}{2}|\omega_1|^2 = 2
\end{equation}
for the vertical direction and 
\begin{equation}
	\Ric_{g_\beta^\alpha}(X,X) \ge 3 - \frac{1}{2} |\omega_1|^2 \ge 1\, ,
\end{equation}
for every horizontal direction $X$.\\

Moreover, by Lemma \ref{lemma: div} and \eqref{eq:curvatureS2S2}, the divergence of $\omega_1$ with respect to the metric $h_\beta^\alpha$ is given by
\begin{equation}
	{\rm div}_N \omega_1 = - 2k \frac{f^\alpha((\beta-1)k)}{f^\alpha( k)} \alpha \sin \psi\, F_2^{\alpha,\beta,*} \, .
\end{equation}
Hence, the only off-diagonal component of the Ricci tensor, $\Ric_{g_\beta^\alpha}(U,F_2^{\alpha,\beta})$, can be made as small as we wish provided $\alpha$ is small enough.

This completes the proof of Proposition \ref{prop:roundS2}. $\qed$\\

\subsection{Trivializing the connection of $\pi : N \to S^2$ for $k$ even}\label{subsec:flattenconnection}

After the application of Proposition \ref{prop:roundS2}, $\pi_{(1,k)} : (S^3\times S^2, g_2^\alpha) \to (N, h_2^\alpha)$ has the structure of a Riemannian principal $S^1$-bundle with totally geodesic fibers of length one, and connection form $\eta_1$.\\ 
Its base space $(N,h^{\alpha}_2)$ has the structure of a Riemannian $S^2$-bundle $\pi: (N,h_2^\alpha) \to (S^2, g_{S^2_{1/2}})$ with totally geodesic $S^2$-fibers with round metric $g_{S^2_\alpha}$.\\

An orthonormal frame for $(S^3\times S^2,g_2^{\alpha})$ is given by
\begin{equation}
	U:= Z_1 + k \frac{\partial}{\partial \theta}\, , \quad
	Z_2\, , \quad
	Z_3 \, , \quad
	F_1^\alpha:= \frac{1}{\alpha} \frac{\partial}{\partial \psi} \, , \quad
	F_2^\alpha:= - \frac{1}{\alpha \sin \psi} \frac{\partial}{\partial \theta}\, .
\end{equation}
Above, $U$ is the vertical unit direction for the $S^1$-bundle $\pi_{(1,k)}:(S^3\times S^2)\to N$. Moreover, $F_1^\alpha$, $F_2^\alpha$ are tangent to the fibers of $\pi: N \to S^2$, and $Z_2, Z_3$ are horizontal.\\
We shall denote by $g_{\rm fib}^\alpha$ the induced round metric on the $S^2$-fibers of the Riemannian submersion $\pi:(N,h_2^{\alpha})\to S^2_{1/2}$.\\

For the remainder of this section we are going to assume that $k\in\mathbb{Z}$ is even. 
Under this assumption we understand from Lemma \ref{lemma:diffeostruct} that $\pi:N\to S^2$ is isomorphic (as an $S^2$-bundle) to the trivial $S^2$-bundle $\pi':S^2\times S^2\to S^2$.\\

Our next goal is to modify the connection of the base space $\pi:N\to S^2$ until it becomes flat. 
Once the connection of the Riemannian submersion $\pi:N\to S^2$ has become flat, $N$ will be isometric to the product of two round spheres. If we maintain the connection and the length of fibres for the $S^1$-bundle $\pi_{(1,k)}:S^3\times S^2\to N$ fixed, it will follow that the end metric is equivariantly isometric to $\pi_{(1,0)}:(S^3\times S^2,g_{S^3_{1}\times S^2_{1/2}})\to (S^2\times S^2,g_{S^2_{1/2}\times S^2_{1/2}})$, as we claimed.

\begin{proposition}\label{prop:flattenconnection}
Let us assume that $k\in\mathbb{Z}$ is even.
Provided that $\alpha<\alpha(k)$, there exist smooth families of Riemannian metrics $(S^3\times S^2,g_{\beta}^{\alpha})$ and $(N,h^{\alpha}_{\beta})$ for $\beta\in [2,3]$ such that the following hold:
\begin{itemize}
\item[i)] $\Ric_{g^{\alpha}_{\beta}},\Ric_{h^{\alpha}_{\beta}}>0$ for any $\beta\in[2,3]$;
\item[ii)] $\pi:(N,h^{\alpha}_{\beta})\to (S^2, g_{S^2_{1/2}})$ is a Riemannian submersion with totally geodesic and round $S^2$-fibers for every $\beta\in[2,3]$. For $\beta=3$, the induced connection is flat and $(N,h^{\alpha}_{3})$ is isomorphic (as a Riemannian $S^2$-bundle) to
\begin{equation}
	\pi_1:(S^2\times S^2,g_{S^2_{1/2}\times S^2_{1/2}})\to (S^2, g_{S^2_{1/2}})\, ,
	\quad \pi_1(x,y) := x \, ;
\end{equation}
\item[iii)] $\pi_{(1,k)}:(S^3\times S^2,g_{\beta}^{\alpha})\to (N,h^{\alpha}_{\beta})$ is a Riemannian principal $S^1$-bundle with totally geodesic fibers and principal connection $\eta_1$, for any $\beta\in[2,3]$. For $\beta=3$, $\pi_{(1,k)}:(S^3\times S^2,g_{3}^{\alpha})\to (N,h^{\alpha}_{3})$ is equivariantly isometric to $\pi_{(1,0)}:(S^3\times S^2,g_{S^3_1\times S^2_{1/2}})\to (S^2\times S^2,g_{S^2_{1/2}\times S^2_{1/2}})$.
\end{itemize}
\end{proposition}

For the remainder of the section, we will be concerned with the proof of Proposition \ref{prop:flattenconnection}. We will assume throughout that $\alpha<\alpha(k)$ so that all the previous steps of our construction apply.\\

Let us define the family of Riemannian metrics $h^{\alpha}_{\beta}$ and $g_{\beta}^{\alpha}$ so that conditions (ii) and (iii) are satisfied. In the next two subsections we will prove that also condition (i) above is met, i.e. the Ricci curvatures of these Riemannian metrics are positive, after a suitable choice of some parameters. \\

By Lemma \ref{lemma:diffeostruct}, $\pi:N\to S^2$ admits a flat Ehresmann connection, that we shall denote by $\Phi_3$. We denote by $\Phi_2$ the Ehresmann connection on $\pi:N\to S^2$ induced by the Riemannian submersion $\pi:(N^{\alpha}_{2}, h_2^\alpha)\to (S^2,g_{S^2_{1/2}})$.\\

We denote by $\Phi_{\beta}$, $\beta\in[2,3]$ the affine combination of $\Phi_2$ and $\Phi_3$. Notice that $\Phi_{\beta}$ is an Ehresmann connection for any $\beta\in [2,3]$, as the fiber of the bundle is compact.\\ 
We introduce a smooth and positive function $\delta:[2,3]\to (0,\infty)$, which we shall specify later in the construction in order to obtain positively Ricci curved metrics. We assume that $\delta(2)=1$ and $\delta(3)=1/(2\alpha)$.\\

For any $\beta\in [2,3]$ we let $h^{\alpha}_{\beta}$ be the unique Riemannian metric on $N$ such that $\pi:(N,h^{\alpha}_{\beta})\to (S^2,g_{S^2_{1/2}})$ is a Riemannian submersion with: 
\begin{itemize}
\item[a)] totally geodesic fibers with induced metric $\delta(\beta)^2 g_{S^2_{\alpha}}$; 
\item[b)] induced connection $\Phi_{\beta}$.
\end{itemize}
Moreover, we let $g^{\alpha}_{\beta}$ be the unique Riemannian metric on $(S^3\times S^2)$ such that $\pi_{(1,k)}:(S^3\times S^2,g^{\alpha}_{\beta})\to (N,h^{\alpha}_{\beta})$ is a Riemannian submersion with totally geodesic fibers of length $2\pi$ and principal connection $\eta_1$.\\

By the very construction, $\pi:(N,h^{\alpha}_3)\to (S^2, g_{S^2_{1/2}})$ is isomorphic as a Riemannian fiber bundle to $\pi_1:(S^2 \times S^2, g_{S^2_{1/2}\times S^2_{1/2}})\to (S^2, g_{S^2_{1/2}})$.
	This means that there exists an isometry $ \Psi : (N,h^{\alpha}_3) \to (S^2 \times S^2, g_{S^2_{1/2}\times S^2_{1/2}})$ such that $\pi_1 \circ \Psi = \pi$.
	
	Then the $S^1$-bundles $\pi_{(1,k)}: S^3 \times S^2 \to N$ and $\Psi^* \pi_{(1,0)} : S^3 \times S^2 \to N$ are isomorphic as they arise from the same cohomology class. In particular, there exists an $S^1$-equivariant diffeomorphism $\hat \Psi : S^3 \times S^2 \to S^3 \times S^2$ with
	$\hat \Psi(\theta \cdot_{(1,k)}(s_1,s_2)) = \theta \cdot_{(1,0)}(\hat \Psi(s_1,s_2))$ whose induced mapping on the quotient is given by $\Psi : N \to S^2 \times S^2$.\\	
	We claim that, up to composition with a Gauge transformation, $\hat \Psi$ is an isometry. It is enough to check that $\hat \Psi$ pulls back $\eta_c$ to $\eta_1$, where $\eta_c$ is the Hopf connection on the first factor in $S^3 \times S^2$.

	As the connection induced by $\pi$ is flat by the very construction, the principal $S^1$ connection $\eta_1$ is a Yang-Mills connection for $\pi_{(1,k)}:(S^3\times S^2,g^{\alpha}_{3})\to (N,h^{\alpha}_{3})$, by Lemma \ref{lemma: div}. Hence $d \hat \Psi^* \eta_c = d\eta_1$, since both forms are Hodge harmonic and $\Psi$ is an isometry. In particular, $\hat \Psi^* \eta_c$ and $\eta_1$ differ by the differential of a smooth function $h: N \to \mathbb{R}$.\\ 
	We can assume $h=0$ by composing $\hat \Psi$ with a suitable Gauge transformation.
It is then clear that (iii) holds. Hence our only remaining goal is to establish positivity of the Ricci curvatures.

\subsubsection{The Ricci curvature of $(N, h_\beta^\alpha)$}

We prove that, up to choosing the $S^2$-fibers' size $\delta(\beta)$ small enough in the interior of the interval $[2,3]$, we have that $\Ric_{h_\beta^\alpha} \ge 3$ for every $\beta\in [2,3]$.\\

From Corollary \ref{cor:canonicalvariation}, we can estimate
\begin{align}
	&\Ric_{h_{\beta}^\alpha}(V,V) \ge  \, \frac{1}{\alpha^2 \delta(\beta)^2}\, ,
	\\
	&\Ric_{h_{\beta}^\alpha}(X,V)=  \, \delta(\beta) g_{S^2_{\alpha}}\left({\rm div}_{S^2_{\alpha}}A_\beta[X] , V \right)\, , 
	\\
	&\Ric_{h_{\beta}^{\alpha}}(X,X)= \, 4 -2 \delta(\beta) g_{S^2_{\alpha}}\left((A_\beta)_X,(A_\beta)_X\right)\, ,
\end{align}
for unit horizontal directions $X$, and vertical directions $V$.

As $\alpha$ is fixed, it is clear that $g_{S^2_{\alpha}}\left({\rm div}_{S^2}A_\beta[X], V \right)$, $g_{S^2_{\alpha}}\left((A_\beta)_X,(A_\beta)_X\right)$ are uniformly bounded on unitary vectors. Hence, we can choose $\delta(\beta)$ small enough in the interior of the interval $[2,3]$ (depending on $\alpha$) to obtain the sought conclusion.

\subsubsection{The Ricci curvature of $(S^3\times S^2, g_\beta^\alpha)$}\label{subsubsec: Ric h trivializing}

Let us check that $\Ric_{g_\beta^\alpha} \ge 2$ for every $\beta\in [2,3]$, which will complete the proof of Proposition \ref{prop:flattenconnection}.\\

By \eqref{eq:RicS^1warped} and the estimates of the previous subsection, we can write
\begin{align}\label{eq:Ric trivilizing}
	& \Ric_{g^{\alpha}_{\beta}}(U,U) = 2\, ,
	\\
\label{eq:Ric trivilizingoffdiag}	& \Ric_{g^{\alpha}_{\beta}}(U,X) = - \frac{1}{2}  {\rm div}_N \omega_1[X]\, ,
	\\
\label{eq:Ric trivilizingbase}	& \Ric_{g^{\alpha}_{\beta}}(X,X) \ge 3 - \frac{1}{2}|\omega_1[X]|^2 \, ,
\end{align}
for every unitary horizontal vector field $X$. 
The delicate point that we need to address in order to establish the claimed positivity is that ${\rm div}_N \omega_1$ and $|\omega_1[X]|^2$ depend on the geometry of $(N, h_\beta^\alpha)$, which is no longer completely explicit.

\medskip

Let us begin with $|\omega_1[X]|^2$. As observed in \eqref{eq:omega_1 vol}, we have
\begin{equation}
	\omega_1 = 2 \pi^* {\rm Vol_{S^2_{1/2}}} \, .
\end{equation}
Since the metric on the base space $S^2$ has not been changed, we have
\begin{equation}
	|\omega_1|^2 = 4 |{\rm Vol_{S^2_{1/2}}}|^2 = 4 \, ,
\end{equation}
independently of $\beta\in[2,3]$.\\

In order to understand ${\rm div}_N \omega_1={\rm div}_{(N,h^{\alpha}_{\beta})}\omega_1$, we apply Lemma \ref{lemma: div} to deduce that
\begin{align}
	{\rm div}_N \omega_1[X] = 
	2 \delta(\beta) g_{\rm fib}^\alpha( A(X_1^\beta,X_2^\beta), X)\, , 
\end{align}
when $X$ is tangent to the fibers of $\pi : N \to S^2$, and ${\rm div}_N \omega_1[V] = 0$ otherwise.\\
In particular, the cross term in \eqref{eq:Ric trivilizingoffdiag} can be made as small as we wish by choosing $\delta(\beta)$ small enough in the interior of $[2,3]$. 
This shows that the Ricci curvature of $(S^3\times S^2,g^{\alpha}_{\beta})$ is indeed positive for any $\beta\in [2,3]$ provided that $\delta(\beta)$ is small enough in the interior of the interval $[2,3]$, as we claimed. $\qed$

\subsection{A more explicit family of diffeomorphisms}\label{subsection:isotopy class}

We explain how to modify the construction in the previous subsections in order to obtain a diffeomorphism $\phi_{2k} : S^3 \times S^2 \to S^3 \times S^2$ isotopic equivalent to $\psi: S^3 \times S^2 \to S^3 \times S^2$ satisfying 
\begin{equation}\label{eq:z1}
	(s_1,s_2) \to (s_1, \psi_{s_1}(s_2)) \, , \quad \psi_{s_1}\in O(3)\, , \, \, \forall\, s_1\in S^3\, .
\end{equation}
This further property was helpful in order to understand the diffeomorphism type of the universal cover for our counterexamples in Section~\ref{sec:univdiffeo}.\\

We start from the explicit action twisting diffeomorphisms that we built in \cite{BrueNaberSemolaMilnor1}.
Let $u,z\in S^3$. We write $u=(u_1,u_2)$, $z=(z_1,z_2)$, where $u_1,u_2,z_1,z_2\in \mathbb C$. 
Set
\begin{align}
	\Phi_k\big(u_1,u_2,z_1,z_2\big) &:= \Big(u_1,u_2, \frac{1}{\sqrt{|u_1|^{2k} + |u_2|^{2k}}}(\bar u_1^k, -u_2^k)\cdot (z_1,z_2)\Big)\notag\\
	&=\Big(u_1,u_2, \frac{1}{\sqrt{|u_1|^{2k} + |u_2|^{2k}}}(\bar u_1^kz_1+u_2^k\bar z_2,-u_2^k\bar z_1+\bar u_1^k z_2)\Big)\, ,
\end{align}
where $\cdot$ denotes the product of $S^3$ as Lie-group.
With this choice, we have the  equivariance property 
\begin{align}
	\Phi_k\big(\theta\cdot_{(1,k)}(u_1,u_2,z_1,z_2)\big) = \theta\cdot_{(1,0)}\Phi_k(u_1,u_2,z_1,z_2)\, .
\end{align}
Moreover, $\Phi_k$ is equivariant with respect to the natural right $S^3$-action on the second $S^3$-factor, i.e.
\begin{equation}
	\Phi_k(u, z \cdot g) = \phi_k(u,z)\cdot g
	\, , \quad
	\text{for every $g\in S^3$}\, .
\end{equation}
So, we can quotient by the right Hopf-action in the second $S^3$ factor obtaining a diffeomorphism satisfying the equivariance
\begin{equation}
	\hat \Psi(\theta\cdot _{(1,2k)}(s_1,s_2)) 
	= \theta\cdot_{(1,0)}\hat{\Psi}(s_1,s_2)\, ,
	\quad s_1\in S^3\, , \, s_2\in S^2\, .
\end{equation}
In particular, the quotient map
\begin{equation}
	\Psi : N \to S^2\times S^2 \, ,
\end{equation}
is an isomorphism of $S^2$-bundles over $S^2$. 

Let $\Phi_3 = \Psi^* \Phi_{\rm flat}$, where $\Phi_{\rm flat}$ is the flat Ehresmann connection in $S^2\times S^2$. With this choice of $\Phi_3$ in subsection \ref{subsec:flattenconnection}, we conclude that $\Psi: (N,h_3^\alpha) \to (S^2, g_{S^2_{1/2}})$ is an isometry. So, following the argument in the last part of subsection \ref{subsec:flattenconnection}, we conclude that the twisting diffeomorphism $\phi_{2k}: S^3\times S^2 \to S^3 \times S^2$ is obtained by lifting $\Psi$.

In particular, $\phi_{2k} \circ \hat \Psi^{-1} : S^3 \times S^2 \to S^3 \times S^2$ is a Gauge transformation of $S^3\times S^2$ thought of as an $S^1$-bundle with respect to the left Hopf-action in the $S^3$-factor. Lemma \ref{lemma:Gauge} ensures that $\phi_{2k} \circ \hat \Psi^{-1}$ is isotopic to the identity. So, $\phi_{2k}$ is isotopic to $\Psi$. It is easy to check that $\Psi$ has the sought structure \eqref{eq:z1}.\\

\end{document}